\documentclass{amsart}
\usepackage[colorlinks,citecolor=red,linkcolor=violet]{hyperref}

\usepackage{graphicx}
\usepackage[english]{babel}

\usepackage{enumerate}
\usepackage{enumitem}
\usepackage{amsmath}
\makeatletter

\@addtoreset{equation}{section}

\@addtoreset{table}{section}

\makeatother

\newtheorem{thm}{Theorem}[section]
\newtheorem{lem}[thm]{Lemma}
\newtheorem{prop}[thm]{Proposition}

\theoremstyle{definition}
\newtheorem{defn}[thm]{Definition}
\newtheorem{nota}[thm]{Notation}

\theoremstyle{remark}
\newtheorem{rem}[thm]{Remark}

\newcommand{\N}{\mathbb{N}}
\newcommand{\Z}{\mathbb{Z}}
\newcommand{\R}{\mathbb{R}}
\newcommand{\C}{\mathbb{C}}

\newcommand{\AAz}{\mathbf{A}}
\newcommand{\AAA}{\mathbf{A}_\nu}
\newcommand{\AAT}{\mathbf{A}_{\nu,\ang}}
\newcommand{\curl}{\mathrm{curl}}
\newcommand{\ang}{\tau}

\newcommand{\EH}{\widetilde{\mathscr{E}}}

\newcommand{\T}{\tan\nu}
\newcommand{\TT}{\tan^2\nu}
\newcommand{\f}{\varphi}

\newcommand{\fmine}{\underline{\tilde\f}}

\newcommand{\sH}{\widetilde{\mathcal{S}}}

\newcommand{\THfder}{\tilde{\f}}
\newcommand{\Lv}{L(\nu)}
\newcommand{\LLv}{\mathcal{L}(\nu)}

\newcommand{\LLvpTT}{\mathcal{L}^{\mathsf{per}}_\ang(\nu,\theta)}
\newcommand{\QQvpT}{\mathcal{Q}^{\mathsf{per}}_{\nu,\theta,\ang}}

\DeclareMathOperator{\supp}{\mathrm{supp}}

\newcommand{\LLvpDLT}{\mathcal{L}^{\mathsf{per},\mathsf{Dir}}_{\ang,L}(\nu,\theta)}
\newcommand{\LLvpDzT}{\mathcal{L}^{\mathsf{per},\mathsf{Dir}}_{\ang,0}(\nu,\theta)}
\newcommand{\QQvpDzT}{\mathcal{Q}^{\mathsf{per},\mathsf{Dir}}_{\nu,\theta,\ang}}

\newcommand{\cv}{\cos\nu}
\newcommand{\sv}{\sin\nu}
\newcommand{\zv}{\zeta(\nu)}
\newcommand{\spe}{\sigma}
\newcommand{\phid}{\phi^{2\Dim}_\nu}
\newcommand{\phidd}{\phi^{3\Dim}_\nu}
\newcommand{\phiddt}{\phi^{3\Dim}_{\nu,\ang}}
\newcommand{\Rh}{\mathbb{R}^2_+}
\newcommand{\RRh}{\mathbb{R}^3_+}

\newcommand{\latti}{w}
\newcommand{\LLat}{\Lambda}
\newcommand{\Lat}{\mathcal{D}}

\newcommand{\Eper}{E^{\mathsf{per}}}

\newcommand{\psivT}{\psi_{\nu,\theta,\ang}^{R,R'}}
\newcommand{\psiN}{\Psi_n}
\newcommand{\psivTt}{\tilde{\psi}_{\nu,\theta,\ang}^{R,R'}}
\newcommand{\dpo}{:}
\newcommand{\bfr}{\mathfrak{b}}
\newcommand{\fou}{f}
\newcommand{\Cfon}{\mathcal{C}^0}
\newcommand{\Cfonbis}{\mathcal{C}^0}

\newcommand{\DCper}{\Lat_{R,R',\theta}}
\newcommand{\HCper}{H^1_\ang(\nu,R,R',\theta)}

\newcommand{\ZCper}{\zeta_\ang(\nu,R,R',\theta)}

\newcommand{\Dim}{D}
\newcommand{\qe}{-}
\newcommand{\qeb}{+}

\newcommand{\NAAT}{(-i\nabla+\AAT)}
\newcommand{\q}{\quad}
\newcommand{\qq}{\qquad}

\newcommand{\eq}{\begin{equation}}
\newcommand{\eeq}{\end{equation}}
\newcommand{\GL}{GL }
\newcommand{\nn}{j}
\newcommand{\mm}{k}
\def\p{\partial}

\usepackage{enumerate, moreverb, textpos, xcolor}
\usepackage{array}

\usepackage{mathrsfs}

\title[Surface Superconductivity]{Concentration Behavior and Lattice Structure of Surface Superconductivity}
\author{S\o ren Fournais}
\author{Jean-Philippe Miqueu}
\author{Xing-Bin Pan}

\address[Fournais and Miqueu]{Aarhus University, Department of
  Mathematics, Ny Munkegade 118, 8000 Aarhus C, Denmark}
\email{fournais@math.au.dk}
\email{jean-philippe.miqueu@math.au.dk}
\address[Xingbin Pan]{Department of Mathematics, East China Normal University, and\\ NYU-ECNU Institute of Mathematical Sciences at NYU Shanghai\\ Shanghai 200062, People's Republic of China}
\email{xbpan@math.ecnu.edu.cn}

\begin{document}
\renewcommand\thepart{\Roman{part}}
\renewcommand*\partname{}

\renewcommand{\abstractname}{Abstract}

\begin{abstract}
We study the three-dimensional Ginzburg-Landau model of superconductivity for strong applied magnetic fields varying between the second and third critical fields. In this regime, it is known from physics that superconductivity should be essentially restricted to a thin layer along the boundary of the sample.
This leads to the introduction of a Ginzburg-Landau model on a half-space.
We prove that the non-linear Ginzburg-Landau energy on the half-space with constant magnetic field is a decreasing function of the angle $\nu$ that the magnetic field makes with the boundary. In the case when the magnetic field is tangent to the boundary ($\nu=0$), we show that the energy is determined to leading order by the minimization of a simplified 1D functional in the direction perpendicular to the boundary. We also study the geometric behavior of the order parameter near the surface of the sample by constructing  formal solutions with lattice properties.
\end{abstract}

\maketitle

\section{Introduction}

\subsection{The Ginzburg-Landau model}
The Ginzburg-Landau (GL) theory of superconductivity was first introduced in the '50s as a
phenomenological macroscopic model \cite{GL}. It successfully describes the behavior of a superconductor subject to an external magnetic field and was later justified by Gor'kov \cite{Gorkov} %
as emerging from the
microscopic Bardeen-Cooper-Schrieffer (BCS) theory. 
This has recently been proved rigorously \cite{FHSS}. 
It has been widely used in the physics literature, 
for instance for successfully predicting the response of superconducting materials to an external magnetic field. Also in the celebrated work of Abrikosov \cite{Abri}, this theory predicted the existence of type II superconductors - in particular, of vortex lattices - before they had been experimentally realized, see \cite{DG66} for a review of this topic for which A. Abrikosov was awarded the Nobel Prize, and the first discussion by Saint-James and de Gennes 
of the surface superconductivity phenomenon that is the subject of this paper. 

In the GL theory, the superconducting state of a sample is described by a
complex-valued wave function $\psi:\R^3\rightarrow \mathbb{C}$ (the order parameter) and a vector field (magnetic potential) $\AAz: \R^3 \rightarrow \R^3$ such that the pair
$(\psi,\AAz)$ is a critical point of a specific energy functional. The interpretation of
$\psi$ and $\AAz$ is explained by the BCS theory as follows:
$|\psi|^2$ is proportional to the density of superconducting particles (the so-called Cooper pairs)
and $\kappa H \curl \AAz$ measures the induced magnetic field inside the sample, with $\kappa>0$ a physical characteristic of the material, and $H$ measuring the intensity of the external magnetic field, that we assume to be constant throughout the sample.
We shall be concerned with type-II superconductors, characterized by $\kappa>\frac{1}{\sqrt{2}}$, and more precisely with the limit $\kappa\rightarrow \infty$ (extreme type-II).

The modulus of the order parameter $|\psi|$ varies between $0$ and $1$: the vanishing of $\psi$ in a certain region or point implies a loss of superconductivity there, due to the absence of Cooper pairs, whereas if $|\psi|=1$ somewhere all the electrons are arranged in Cooper pairs and thus superconducting. The cases $|\psi|=1$ and $|\psi|=0$ everywhere in $\Omega$ correspond to the so-called
perfectly superconducting and normal states, known to be preferred for small and large applied field respectively. When $|\psi|$ is not identically $0$ nor $1$, for intermediate values of the applied field, one says that the system is in a mixed state.

The behavior of a type-II superconductor is distinguished by three critical values of the intensity
of the applied magnetic field which we denote by $H_{C_1}$, $H_{C_2}$ and $H_{C_3}$. These critical fields may be described in terms of the wave function $\psi$ as follows. 
When the external magnetic field strength $H$ satisfies $H < H_{C_1}$, the material is in the superconducting phase, which corresponds to $|\psi|>0$ everywhere. The sample stays in the superconducting state until the first critical field is reached. When $H_{C_1} < H < H_{C_2}$, the magnetic field penetrates the sample in quantized vortices. 
These vortices correspond to isolated zeros of $\psi$; 
their number increases with the increase of the strength of the external field $\kappa H$. In the $2$D case, they arrange themselves on a triangular lattice, the famous Abrikosov lattice, that survives until a second critical value of the field is reached. When $H_{C_2} < H < H_{C_3}$ superconductivity is confined to (part of) the surface of the sample corresponding to $|\psi|$ very small in the bulk.
More precisely the \GL order parameter is exponentially decaying far from the boundary. This is the surface superconductivity regime. Finally, when $H > H_{C_3}$, superconductivity is lost, which is reflected by $\psi= 0$ everywhere, and the normal state becomes the global minimizer of the \GL energy.

In the last decades, much progress has been made towards establishing the aforementioned behavior of type-II superconductors by studying minimizers of the \GL energy. The monograph \cite{SS07} and references therein contains an analysis of vortices and the critical field $H_{C_1}$. Concerning the analysis of the critical fields $H_{C_2}$ and $H_{C_3}$ we mention \cite{FK13, FH} (and references therein). As one can see in \cite{FH, SS07}, the \GL model has a rich mathematical structure whose analysis requires a diversity of techniques, many of which have been developed especially for the study of the model. While a detailed study of the \GL model in a two dimensional domain has been the subject of numerous papers, the study of the model in a three dimensional domain is much less developed.

\subsection{Objective of the paper}

Our study is motivated by the mathematical theory of the surface superconductivity of $3$D samples. 
It is well understood  \cite{P02, P04, FKP13} that in a suitable range of magnetic field strengths the solutions to the \GL equations are localized near the boundary. We want to improve the understanding of this boundary layer.
For this purpose, we examine the energy contribution of the order parameter in the vicinity of the domain boundary.
We believe that in the surface superconducting state, the order parameters will exhibit a certain lattice structure in the vicinity of the domain boundary similar to the Abrikosov lattices of $2$D samples. Understanding the lattice structure near the boundary will help us to understand the vortex lattices of superconductivity when the applied magnetic field decreases and approaches the second critical field $H_{c_2}$.

After rescaling and taking limits, the behavior of the solutions to the \GL equations in a boundary layer can be understood from the  limiting equations in $\R^3_+$. So we consider only the problem in $\R^3_+$.

\begin{rem}[The spectral quantity $\Theta_0$]\label{QTheta}$\,$\\
The important constant $\Theta_0$ has already been mentioned in \S 1.1 and will appear in the statements. Consider the harmonic oscillator $H(\xi)$ defined for all $\xi\in\mathbb{R}$ on the half-axis $\mathbb{R}_+$ as follows\dpo
\begin{equation}\label{DeGennes}
H(\xi)=-\frac{d^2}{dt^2}+(t-\xi)^2 \quad \text{in} \ L^2(\mathbb{R}_+),
\end{equation}
with Neumann boundary condition $u'(0)=0$. This operator has compact resolvent and it follows from Sturm-Liouville theory that its eigenvalues are simple. Let $\mu_1(\xi)$ denote the first eigenvalue of $H(\xi)$. The constant $\Theta_0$ is defined as\dpo
\begin{equation}\label{eq-th0}
\Theta_0=\underset{\xi\in\mathbb{R}}{\inf}\ \mu_1(\xi),
\end{equation}
see for instance \cite{LP00b} and also \cite{HM01} 
for the mathematical analysis of $\Theta_0$.
\end{rem}

\subsection{The $3$D surface energy}\label{Sec:Ini}
In this paper, for all $m\in\{1,2,3\}$, we denote by $x_j$ ($j\in \{1, \cdots, m\}$) the coordinates of $x\in \mathbb{R}^m$. We define $\R^m_+=\{x\in\R^m: x_1>0\}$ and $\partial\R^m_+=\{x\in \R^m: x_1=0\}$ the boundary of $\R^m_+$.

Let $\nu\in[0,\frac{\pi}{2}]$, and $\ell>0$. We introduce the set\dpo
\begin{equation}\label{Dinitial}
\mathcal{D}_\ell=(0,\infty)\times(-\ell,\ell)\times(-\ell,\ell),
\end{equation}
and the magnetic potential $\AAA$ defined on $\mathbb{R}^3_+$ by
\begin{equation}\label{Av}
\mathbf{A}=\mathbf{A}_\nu=\left(\begin{array}{c} 0\\ 0 \\ -x_1 \cv+x_2 \sv\end{array}\right),
\end{equation}
for which the associated magnetic field is the constant unit vector that makes an angle $\nu$ with the $x_2x_3$ plane\dpo
\begin{equation}\label{Bv}
\mathbf{B}=\mathbf{B}_\nu= \nabla \times \mathbf{A}_\nu=\left(\begin{array}{c} \sv\\ \cv \\ 0\end{array}\right).
\end{equation}

\begin{rem}
By the standard gauge invariance arguments our energy (in particular, the energy functional $\mathscr{E}_{\bfr,\nu,\ell}$ defined below) depends on the magnetic field $\mathbf{B}_\nu$ but not on the specific choice of vector potential $\mathbf{A}_\nu$ with $\nabla \times \mathbf{A}_\nu = \mathbf{B}_\nu$. We only fix this choice for concreteness.
\end{rem}

\begin{defn}\label{def:EnergyInitial}
We consider the following reduced \GL type energy functional\dpo
\begin{equation}\label{GL}
\mathscr{E}_{\bfr,\nu,\ell}(\f)=\displaystyle{\int_{\mathcal{D}_{\ell}}\left(|(-i\nabla+\mathbf{A}_\nu)\f|^2-\bfr|\f|^2+\frac{\bfr}{2}|\f|^4\right)\,\mathrm{d}x},
\end{equation}
for $\f$ in the space\dpo
\begin{equation}\label{spaceSInitial}
\mathcal{S}_{\ell}=\left\{\f \in L^2(\mathcal{D}_{\ell})~:~ (-i\nabla+\mathbf{A}_\nu)\f\in L^2(\mathcal{D}_{\ell},\mathbb{C}^3),\;\; \f=0 \;\; \text{on} \;\; \partial\mathcal{D}_{\ell}\backslash \{x_1=0\}\right\}.
\end{equation}
Furthermore, we define\dpo
\begin{align}\label{eq:E}
E(\bfr,\nu,\ell)= \inf_{\f \in \mathcal{S}_{\ell}} \mathscr{E}_{\bfr,\nu,\ell}(\f),
\end{align}
and, for those values of $\bfr$ where the limit exists,
\begin{align}\label{eq:EnergyPerArea}
e(\bfr,\nu)= \lim_{\ell\rightarrow \infty} \frac{1}{4\ell^2} E(\bfr,\nu,\ell).
\end{align}
\end{defn}

\begin{rem}\label{rem:existenceInitial}
The existence of the limit \eqref{eq:EnergyPerArea}
was proved under the restriction $\bfr \in [\Theta_0,1]$ in \cite{FKP13}.
More precisely, it has been proved that (see \cite[Theorem 3.13]{FKP13}) for all $\nu\in\left[0,\frac{\pi}{2}\right]$, the function $\bfr\mapsto e(\bfr,\nu)$ is continuous and monotone decreasing, and that for all $\bfr\in\left[\Theta_0,\nu\right]$, the function $\nu\mapsto e(\bfr,\nu)$ is continuous.
For $\bfr\le \Theta_0$ it is clear that $e(\bfr,\nu)=0$ 
and for $\bfr>1$ there is no boundary concentration so the limit in \eqref{eq:EnergyPerArea} does not exist. It is not stated explicitly in \cite{FKP13} but one can easily show that for all $\nu\in\left[0,\frac{\pi}{2}\right]$, the function $\bfr\mapsto e(\bfr,\nu)$ is also concave (as an infimum and a limit over affine functions).
\end{rem}

\subsection{Statement of the results}\label{main_results}
For $\nu=0$, we have a complete understanding of the limit $e(\bfr,0)$ in \eqref{eq:EnergyPerArea}.
For the similar problem in $2$D it was proposed in \cite{P02} and proved in \cite{CR14} that a non-linear Ansatz with separation of variables is correct for the ground state. We will see that the $3$D case is completely analogous (with essentially the same proof).

\begin{thm}\label{thm:Uniqueness}
For $\nu=0$ and $\bfr\in (\Theta_0,1]$
we have
\begin{equation}\label{eq:1D}
e(\bfr,\nu=0) = E_0^{1\Dim},
\end{equation}
where $E_0^{1\Dim}$ is defined by
\begin{equation}\label{E1D}
E_0^{1\Dim}=\inf_{\xi \in \R}\left(\inf_{f\in H^1(\R_{+})}
\mathscr{E}_{\bfr,\xi}^{1\Dim}(f)\right),
\end{equation}
and where
\begin{equation}\label{eq:1DGL_intro}
\mathscr{E}_{\bfr,\xi}^{1\Dim}(f):=\int_0^{\infty} \Big\{
|f'(t)|^2 + (t-\xi)^2 |f(t)|^2 -\bfr |f(t)|^2 + \frac{\bfr}{2} |f(t)|^4 \,\Big\} \mathrm{d}t.
\end{equation}
\end{thm}
We note that the infimum is both taken with respect to the function $f$ and the real number $\xi$. Minimizing the $1$D-functional \eqref{eq:1DGL_intro} with respect to $f$, we obtain an energy $E_{\bfr,\xi}^{1\Dim}$ and a minimizer $f_{\bfr,\xi}=f_{\xi}$. Then, minimizing $E_{\bfr,\xi}^{1\Dim}$ with respect to $\xi$ gives a minimal energy $E_{\bfr,0}^{1\Dim}=E_{0}^{1\Dim}$ and a minimizer $\xi_0$.
The proof of Theorem~\ref{thm:Uniqueness} is similar to the $2$D-case and is given in Section~\ref{Sec:Parallel}.

\begin{thm}\label{thm:mon}
For all $\bfr\in(\Theta_0,1]$, the function $[0,\frac{\pi}{2}]\ni\nu\mapsto e(\bfr,\nu)$ is monotone non-decreasing.
\end{thm}
The proof of this statement is given in Section \ref{Sec:Generalized}. Theorem \ref{thm:mon} complements the result \cite[Theorem 3.13]{FKP13}.

The monotonicity of the quantity $e(\bfr,\nu)$ with respect to the angle $\nu$ has an interest in the theory of superconductivity. Indeed, the ground state energy is a function of the inclination of the magnetic field and the result gives that the energy increase as the magnetic field tends to be perpendicular to the surface of the sample.

Finally, in Section \ref{Sec:Abri}, we construct bounded solutions with lattice structure,  
in the case when $\nu\neq 0$. 
The construction is completely analogous to the Abrikosov solutions in $2$D.
The result we prove -- Theorem \ref{bislem:3.1Par}, stated in Section \ref{Sec:statement} -- implies an upper bound on the energies \eqref{eq:E} and \eqref{eq:EnergyPerArea} defined in the first section.

\subsection*{Acknowledgements}
Fournais and Miqueu were partially supported by a Sapere Aude grant
from the Independent Research Fund Denmark, Grant number
DFF--4181-00221. Pan was partially supported
by the National Natural  Science Foundation of China grants no. 11671143 and no. 11431005.

\section{The linear problem}
Before starting the analysis of the non-linear Ginzburg-Landau functional $\mathscr{E}_{\bfr,\nu,\ell}$, we describe some of the linear spectral results that are needed.

In all of the paper, we will denote by $\spe(\mathcal{L})$ the spectrum of any given operator $\mathcal{L}$.

First we consider the magnetic Schr\"odinger operator associated with a constant magnetic field in $\R^3$.

\begin{prop}\label{prop:NoBoundary}
Let the magnetic Schr\"{o}dinger operator $(-i\nabla + {\mathbf A}_{\nu})^2$ be defined as a self-adjoint operator on $L^2(\R^3)$ with form domain
$$\left\{\psi\in L^2(\R^3) ~:~ (-i\nabla+\AAA)\psi\in L^2(\R^3,\C^3)\right\}.$$ For all $\nu\in\left[0,\frac{\pi}{2}\right]$ we have
$$
\inf \spe\{ (-i\nabla + {\mathbf A}_{\nu})^2 \} = 1.
$$
\end{prop}

\begin{proof}
This is just the well known structure of the Landau bands for constant magnetic field in $\R^3$.
\end{proof}

Consider the Schr\"{o}dinger operator of a particle moving in a 3-dimensional  half-space $\R^3_+$, subject to a constant magnetic field of unit strength having an angle $\nu$ to the boundary-plane $\p \R^3_+$,
\begin{equation}\label{eq-3D-op-2}
\mathcal L(\nu)=(-i\nabla+\mathbf{A}_\nu)^2\quad{\rm in}\ L^2(\R_+^3)\,,
\end{equation}
with domain
\begin{multline}\label{eq:dom}
D(\mathcal L(\nu))
=\{u\in L^2(\R_+^3)~:~(-i\nabla+\mathbf{A}_\nu) u\in L^2(\R_+^3, \C^3)\,,\\
(-i\nabla+\mathbf{A}_\nu)^2 u\in L^2(\R_+^3),\;\; \partial_{x_1}u=0\text{ on } \p\mathbb R^3_+\}\,.
\end{multline}
The spectrum of the Schr\"{o}dinger operator with  (magnetic) Neumann boundary condition introduced in \eqref{eq-3D-op-2} has been the object of study of several works and is by now well understood. 
We denote by $\zeta(\nu)$ the lowest point in the spectrum of $\mathcal L(\nu)$,
\begin{equation}\label{zeta}
\zeta(\nu)=\inf\spe\big{(}\mathcal L(\nu)\big{)}\,.
\end{equation}
We collect below some properties concerning the quantity $\zeta(\nu)$ (see e.g. \cite[Lemmas~7.2.1 \&  7.2.2]{FH}).

In connection with the analysis of
the operator  $\mathcal L(\nu)$, we introduce the two-dimensional operator
\begin{equation*}
L(\nu)=-\partial_{x_1}^2-\partial_{x_2}^2+(-x_1\cos\nu+ x_2\sin\nu)^2\quad\text{in}\quad L^2(\R^2_+)\,,
\end{equation*}
whose domain $D(L(\nu))$ is
\begin{multline*}
D(L(\nu))=\big\{u\in L^2(\R_+^2)~:~(-x_1\cos\nu+ x_2\sin\nu)^ju\in L^2(\R^2_+),~j=1,2,~
\\ \partial_{x_1}u=0~\text{on} ~ \p\mathbb R^2_+ \big\}\,.
\end{multline*}

\begin{lem}\label{lem-p-z(nu)}
Let $\Theta_0$ be the universal constant introduced in~\eqref{eq-th0}. The function $[0,\pi/2]\ni\nu\mapsto\zeta(\nu)$ is continuous, monotonically non-decreasing, and we have that $\zeta(0)=\Theta_0$ and $\zeta(\pi/2)=1$.
Furthermore, for all $\nu\in(0,\pi/2)$ we have
\begin{enumerate}
\item \label{item1} $\spe(\mathcal L(\nu))=\spe(L(\nu))$\,;
\item $\spe_{\text{ess}}(L(\nu))=[1,\infty)$\,.
\end{enumerate}
\end{lem}

The results of Lemma~\ref{lem-p-z(nu)}  (concerning the spectrum of $\mathcal L(\nu)$) have been obtained in \cite[Theorem 3.1]{LP00}, before being improved in \cite{HM}.
Notice how the dimensional reduction in conclusion (\ref{item1}) in Lemma \ref{lem-p-z(nu)} is only valid for $\nu>0$.

\begin{rem}\label{rem-2D-sp}
Suppose that $\nu\in(0,\pi/2)$. It results from Lemma~\ref{lem-p-z(nu)} that $\zeta(\nu)$ is the lowest eigenvalue of $L(\nu)$. It is a simple eigenvalue by a positivity argument. Consequently, we can select a unique non-negative eigenfunction $\phi\in L^2(\R^2_+)$, normalized and
such that
\begin{equation}\label{uniqueeigen}
\int_{\R^2_+}\left(|\nabla\phi|^2+|(-x_1\cos\nu+x_2\sin\nu)\phi|^2\right)\,\mathrm{d}x
=\zeta(\nu)\,.
\end{equation}
\end{rem}
\begin{nota}\label{resLP00}
The function defined through \eqref{uniqueeigen} will be denoted $\phid$.
\end{nota}
The next result concerns the decay of the function $\phid$ that we will need. We refer to \cite{R10a} for a stronger statement. 

\begin{prop}\label{prop:ptdecay}[See \cite{R10a}]
Let $\nu\in(0,\frac{\pi}{2})$. The ground state $\phid$ of the operator $\Lv$ belongs to the Schwartz class $\mathscr{S}(\overline{\Rh})$.
\end{prop}
For completeness, let us mention that 
other decay properties of the eigenfunction $\phi^{2\Dim}_\nu$ are established in \cite[Theorem~1.1]{BDPR12}. 

In \cite[Theorem 4.2]{LP00}, it is proved that for all $\nu\in(0,\frac{\pi}{2})$, the dimension of the eigenspace associated with the lowest eigenvalue $\zv$ for the operator $\LLv$ is infinite. Thus, we have that $\zv$ is not a discrete eigenvalue but belongs to the essential spectrum of the operator $\LLv$.  
The following result gives the form of the $L^2(\RRh)$ eigenfunctions.
\begin{lem}\label{Sdecay}
Let $\nu \in (0,\frac{\pi}{2})$.
For all $\fou\in L^2(\mathbb{R})$, the function $\phidd$ defined by
\begin{equation}\label{shape}
\phidd(x)=\mathcal{F}^{-1}\left(\xi_3\mapsto \fou(\xi_3)\phid(x_1,x_2-\frac{\xi_3}{\sv})\right),
\end{equation}
(where $\mathcal{F}^{-1}$ is the inverse Fourier transform in the $\xi_3$ variable) is an $L^2(\RRh)$ eigenfunction associated with the eigenvalue $\zv$ of the operator $\LLv$ (in particular, the Neumann condition at the boundary of $\RRh$ is satisfied), and all the $L^2(\RRh)$ eigenfunctions associated with the eigenvalue $\zv$ are in this form.

What is more, for every $\fou \in \mathscr{C}^\infty_c(\mathbb{R})$ (the set of smooth functions with compact support), we have that the function $\phidd$ defined by \eqref{shape} belongs to the Schwartz class $\mathscr{S}(\overline{\RRh})$.
\end{lem}
\begin{proof}
To prove the first part of the statement, we note that the fact that the function $\phidd$ is an eigenfunction associated with the eigenvalue $\zv$ of the operator $\LLv$ comes from the fact that the function $\phid$ is an eigenfunction associated with the eigenvalue $\zv$ of the operator $\Lv$. We refer to the proof of Lemma 4.5 in \cite{LP00} for the details.

We prove here the last assertion of the lemma. 
We will only prove the decay of $\phidd$ since 
the same decay properties of the derivatives can be obtained in the same way, using that $\phid\in \mathscr{S}(\overline{\Rh})$ and $\fou\in \mathscr{C}^\infty_c(\mathbb{R})$. 

Since $\phid\in \mathscr{S}(\overline{\Rh})$ and $\fou$ belongs to $\mathscr{C}^\infty_c(\mathbb{R})$, we easily have that the function $\phidd$ defined by \eqref{shape} belongs to $\mathscr{C}^\infty(\mathbb{R}^3)$.

It suffices to establish the decay in each variable individually. 
The decay in the $x_1$ variable (uniformly in $x_2$) is obvious since the function $\phid$ belongs to  $\mathscr{S}(\overline{\Rh})$ (in the $x_1$ and $x_2$ variables, see Proposition \ref{prop:ptdecay}). The good estimate in the $x_3$ variable is also straightforward since the Fourier transform of a function in $L^1$ is bounded. We now deal with the decay in the $x_2$ variable. Let $k\in \mathbb{N}$. We are going to give an upper bound on the term
\begin{equation*}
\left|x_2^k\displaystyle{\int_{\mathbb{R}_{\xi_3}}\phid\left(x_1,x_2-\frac{\xi_3}{\sv}\right)e^{ix_3\xi_3}\fou(\xi_3)\,\mathrm{d}\xi_3}\right|.
\end{equation*}
Let $\supp f \subset [-M,M]$. 
For convenience, we perform a change of variable in the integral ($\eta_3=x_2-\frac{\xi_3}{\sv}$) and estimate
\begin{align*}
&\left|x_2^k\displaystyle{e^{ix_3x_2\sv}\int_{\mathbb{R}_{\eta_3}}\phid(x_1,\eta_3)e^{-ix_3\eta_3\sv}\fou((x_2-\eta_3)\sv)\,\mathrm{d}\eta_3}\right| \\
&\leq 
|x_2|^k \| f \|_{\infty} \int_{x_2 - M/\sin \nu}^{x_2 - M/\sin \nu}|\phid(x_1,\eta_3)|\,\mathrm{d}\eta_3\\
&\leq \frac{2M}{\sin \nu} \| f \|_{\infty} \frac{|x_2|^k}{(1+|x_2- M/\sin \nu|^2)^N} \sup_{\Rh} \left( (1+\eta^2)^N |\phid(x_1,\eta)| \right),
\end{align*}
which gives the desired bound upon choosing $N\geq k$.
\end{proof}

\section{Parallel field}\label{Sec:Parallel}
In this section we study the case when $\nu=0$ and prove Theorem~\ref{thm:Uniqueness}. We will prove that the `thermodynamic' limit---the limit in \eqref{eq:EnergyPerArea}---can be expressed through the $1$D functional $\mathscr{E}_{\bfr,\xi}^{1\Dim}$  given in \eqref{eq:1DGL_intro}.
The principal properties of the functional $\mathscr{E}_{\bfr,\xi}^{1\Dim}$ are well known (see \cite[Section 3 and Appendix]{P02}, \cite[Section 14.2]{FH} and the references therein, and also \cite{FHP11}). For understanding and completeness, we recall in the following lemmas the underlying results that we will need.
\begin{lem}\label{prelim1D}
For all $\xi\in\mathbb{R}$ and all $\bfr\in\mathbb{R}_+$, the functional $\mathscr{E}_{\bfr,\xi}^{1\Dim}$ admits a non-negative minimizer $f_{\bfr,\xi}$,
in the space
\begin{equation*}
B^1(\mathbb{R}_+)=\left\{f\in L^2(\R_+,\R
)~:~ t^pf^{(q)}(t)\in L^2(\R_+,\R
),\; \forall p, q\in \N,\; p+q\le 1\right\}.
\end{equation*}
The minimizers satisfy the Euler-Lagrange equations
\begin{equation}\label{varf0}
\left\{\aligned
 -&f_{\bfr,\xi}''+(t-\xi)^2f_{\bfr,\xi}=\bfr(1-f_{\bfr,\xi}^2)f_{\bfr,\xi},\q t>0,\\
  &f_{\bfr,\xi}'(0)=0.
 \endaligned\right.
\end{equation}
Moreover, we have
\begin{equation}\label{energyf0}
\underset{f\in B^1(\mathbb{R}_+)}{\inf}\mathscr{E}_{\bfr,\xi}^{1\Dim}(f)=-\frac{\bfr}{2}\Vert f_{\bfr,\xi}\Vert^4_{L^4(\R_+)},
\end{equation}
and the inequality\dpo
\begin{equation*}
\Vert f_{\bfr,\xi}\Vert_{L^\infty(\R_+)}\le 1.
\end{equation*}
What is more, the equation\dpo
\begin{equation}\label{FH14.19}
\left\{\aligned
 -&f''+(t-\xi)^2f=\bfr (1-f^2)f,\q t>0,\\
&f'(0)=0.
\endaligned\right.
\end{equation}
admits non-trivial bounded solutions if and only if $\mu_1(\xi)<\bfr$ (see Subsection \ref{QTheta} (above) for the definition of $\mu_1$), and if $f\in L^\infty(\R_+)$ satisfies \eqref{FH14.19}, then we have that
$$\Vert f\Vert_{L^\infty(\R_+)}\le 1\q\text{\rm and}\q f\in L^2(\R_+).
$$
If we have $\mu_1(\xi)<\bfr<1$, the non-negative minimizer $f_{\bfr,\xi}=f_\xi$ of the functional $\mathscr{E}_{\bfr,\xi}^{1\Dim}$ is unique and strictly positive.
\end{lem}
We should notice that our conventions are slightly different from the ones considered in \cite{FH} (see for instance (3.9) in \cite{FH} and \eqref{DeGennes} in the present paper) and the ones considered in \cite{CR14} (see for instance the choice of $(\nabla+i\mathbf{A})$ for the linear part of the Ginzburg-Landau functional instead of the expression $(-i\nabla+\mathbf{A})$ we consider). The statements we give here have been adapted to our choices.

\begin{nota}\label{f0}
We recall that $f_{\xi_0}$ is defined above Theorem \ref{thm:mon} and we denote $f_0=f_{\xi_0}$.
\end{nota}

The proof of Theorem \ref{thm:Uniqueness} directly follows the approach presented in \cite{CR14} and is made in two steps consisting in obtaining an upper and a lower bound.

\begin{proof}[Proof of Theorem~\ref{thm:Uniqueness}]
We will only prove Theorem~\ref{thm:Uniqueness} for $\bfr\in (\Theta_0,1)$---the extension to $\bfr = 1$ following by continuity on both sides.
The continuity of the map $[0,\frac{\pi}{2}]\ni\nu\mapsto e(\bfr,\nu)$ is given in
\cite[Theorem 3.13]{FKP13} for all $\nu\in[0,\frac{\pi}{2}]$.
The continuity of $E_0^{1\Dim}$ is easier and is left to the reader.

\noindent\textbf{Upper bound.}
For $\nu=0$ and $\bfr\in (\Theta_0,1]$ (here the endpoint $\bfr =1$ can easily be included),
the estimate
\begin{align}\label{eq:1Dupper}
E(\bfr,\nu=0,\ell) \le 4\ell^2 E_0^{1\Dim}+o(\ell)
\end{align}
is obtained by considering the following trial state
\begin{equation}\label{trial1D}
\f(x)=\f(x_1,x_2,x_3)=f_0(x_1)e^{i\xi_0x_3},
\end{equation}
(defined for any $x\in \RRh$) suitably localized in order to satisfy the Dirichlet boundary condition. We omit the details.

\smallbreak
\noindent\textbf{Lower bound.}
The lower bound results from an energy decoupling through a ground state representation. It relies on some preliminary results related to the $1$D-functional $\mathcal{E}_{\bfr,\xi_0}^{1\Dim}$. In the following lemma we gather the preliminary properties that we need, see \cite[Section 3]{P02}, \cite[Section 14.2]{FH} and \cite[Section 3.2]{CR14} for details.
\begin{lem}\label{prelim}
For all $\bfr\in (\Theta_0,1)$, the following properties hold.
\begin{enumerate}[label=(\roman*)]
\item \label{f0pos} The function $f_0$ (introduced in Notation \ref{f0}) is strictly positive everywhere in $\mathbb{R}_+$.
\item \label{F0neg} The function $F_0$ defined by
\begin{equation}\label{F0}
F_0(x_1)=2\displaystyle{\int_0^{x_1}(y-\xi_0)f_0^2(y)}\,\mathrm{d}y,
\end{equation}
satisfies $F_0(+\infty)=0$ and is negative for all $x_1>0$.
\item \label{cost} The `cost function' $K_0$ defined as\dpo
\begin{equation}\label{CR2.18}
K_0(x_1)=f_0^2(x_1)+F_0(x_1),
\end{equation}
is positive on $\R_+$.
\end{enumerate}
\end{lem}
In the particular case when $\nu=0$, the functional $\mathscr{E}_{\bfr,\nu,\ell}$ defined through \eqref{GL} has the following form\dpo
\begin{equation}\label{CR2.12}
\mathscr{E}_{\bfr,\nu=0,\ell}(\f)=\displaystyle{\int_{[-\ell,\ell]^2}\left(\int_0^{+\infty}\left\{|(-i\nabla-x_1\mathbf{e}_{3})\f|^2-\bfr|\f|^2+\frac{\bfr}{2}|\f|^4\right\}\mathrm{d}x_1\right)\mathrm{d}x_2\mathrm{d}x_3,}
\end{equation}
where $\mathbf{e}_{3}$ is the unit vector in the $x_3$ direction. Thanks to the property \ref{f0pos} of Lemma \ref{prelim}, to any function $\f$ we may associate a  function $v$ with the Ansatz\dpo
\begin{equation}\label{CR2.13}
\f(x_1,x_2,x_3)=f_0(x_1)e^{i\xi_0x_3}v(x_1,x_2,x_3).
\end{equation}
By density of the set of functions with compact support  in $\mathcal S_\ell$, it suffices to work with functions satisfying
\begin{equation}\label{condition}
\f(x_1,x_2,x_3)=0, \quad  \text{for} \ x_1 \ \text{sufficiently large}.
\end{equation} Considering \eqref{condition} and using the variational equation \eqref{energyf0} for $f_0$, integration by parts yields,
\begin{equation}\label{CR2.14}
\mathscr{E}_{\bfr,\nu=0,\ell}(\f)=4\ell^2E_0^{1\Dim}+\mathscr{E}_{0}(v),
\end{equation}
where $\mathscr{E}_{0}(v)$ is defined as
\begin{equation}\label{CR2.15}
\mathscr{E}_{0}(v)=\displaystyle{\int_{\mathcal{D}_\ell}
\ f_0^2(x_1)\left\{|\nabla v|^2-2(x_1-\xi_0)\mathbf{e}_{3}\cdot\mathbf{j}(v)+\frac{\bfr}{2}f_0^2(x_1)(1-|v|^2)^2\right\}\,\mathrm{d}x},
\end{equation}
and where $\mathbf{j}(v) = (j_1, j_2, j_3) $ is given by
\begin{equation}\label{supercurr}
\mathbf{j}(v)=\frac{i}{2}(v \overline{\nabla v}-\overline{v}\nabla v).
\end{equation}
The boundary terms vanish because the function $f_0$ satisfies the Neumann boundary condition at $x_1=0$ (and using \eqref{condition}).

Therefore, it suffices to prove the positivity of the reduced functional $\mathscr{E}_0$ given in \eqref{CR2.15}. We define the following field $\mathbf{F}_0=(0,F_0,0)$, with $F_0$ from \eqref{F0}.
We notice that
\begin{equation*}
\left(2 (x_1-\xi_0) f_0^2 \right)\mathbf{e}_{3}\cdot\mathbf{j}(v)=(\partial_{x_1}F_0) (v\overline{\partial_{x_3}v}-\overline{v}\partial_{x_3}v).
\end{equation*}
An integration by parts in the $x_1$ variable on the term involving $\mathbf{j}(v)$ (using \eqref{condition}), yields
\begin{equation*}
\displaystyle{\int_{\mathcal{D}_\ell}-2(x_1-\xi_0)f_0^2(x_1)\mathbf{e}_{3}\cdot \mathbf{j}(v)\,\mathrm{d}x}
=
\displaystyle{\int_{\mathcal{D}_\ell}F_0 \partial_{x_1}j_3(v)\,\mathrm{d}x}.
\end{equation*}
An integration by parts in the $x_3$ variable gives, for each fixed $x_1$ and $x_2$ (and using the Dirichlet boundary condition),
\begin{align*}
\frac{1}{2}\displaystyle{\int_{-\ell}^\ell \partial_{x_1}(v\overline{\partial_{x_3}v}-\overline{v}\partial_{x_3}v)\,\mathrm{d}x_3}
=
-i \int_{-\ell}^\ell \curl\, \mathbf{j}(v)\cdot \mathbf{e}_{2}\,\mathrm{d}x_3.
\end{align*}
Therefore, we have finally obtained
\begin{equation*}
\mathscr{E}_0(v)=\displaystyle{\int_{\mathcal{D}_\ell}
\ \left\{f_0^2(x_1)|\nabla v(x)|^2+F_0(x_1)\mu(v(x))+\frac{\bfr}{2}f_0^4(x_1)(1-|v(x)|^2)^2\right\}\mathrm{d}x,}
\end{equation*}
where $\mu(v)=\curl\, \mathbf{j}(v)\cdot \mathbf{e}_{2}$.
Using that $F_0$ is negative (by Lemma \ref{prelim}, \ref{F0neg}), and $|\mu(v)|\le
\left|\curl\,\mathbf{j}(v)
\right|\le |\nabla v|^2$, we get
\begin{align}
\mathscr{E}_0(v)&\ge\displaystyle{\int_{\mathcal{D}_\ell}
 \ \left(f_0^2(x_1)|\nabla v(x)|^2+F_0(x_1)|\mu(v(x))|\right)\mathrm{d}x} \nonumber \\
\label{CR2.17}
&\ge \int_{\mathcal{D}_\ell}
 \ \left(f_0^2(x_1)+F_0(x_1)\right)|\nabla v(x)|^2\,\mathrm{d}x  \nonumber \\
 &\ge 0,
\end{align}
where the last inequality follows from \ref{cost}
in Lemma \ref{prelim}.
This finishes the proof of Theorem~\ref{thm:Uniqueness}.
\end{proof}

\section{Ground state energy for general direction of the magnetic field}\label{Sec:Generalized}

In this section we will prove Theorem~\ref{thm:mon}. As stated in Lemma~\ref{lem-p-z(nu)} in Section 2, a similar monotonicity result is valid for the linear problem, i.e. for the spectral quantity $\zeta(\nu)$.
The proof of the monotonicity of $\nu \mapsto \zeta(\nu)$ is rather easy once a clever change of variables is implemented. Our proof of Theorem~\ref{thm:mon} is very much inspired by the analysis of the linear problem and uses the same change of variables performed in the proof of the monotonicity of the lowest eigenvalue (see \cite[Theorem 3.1]{LP00} and \cite{HM}).

\subsection{Generalization of the initial problem in $3$D}

Let $\nu\in[0,\frac{\pi}{2}]$, $\ell>0$ and $\ell_3>0$. For $\beta\in[0,\frac{\pi}{2})$, we introduce the set
\begin{equation*}
\mathcal{D}_{\ell,\ell_3,\beta}=\big\{x=(x_1,x_2,x_3)\in\R^3~:~ x_1>0,\;\; |x_2-x_1\tan\beta|<\ell,\;\; |x_3|<\ell_3\big\},
\end{equation*}
for which we have the following correspondence with the set defined in \eqref{Dinitial}
\begin{equation*}
\mathcal{D}_\ell=\mathcal{D}_{\ell,\ell,0}.
\end{equation*}
We consider the magnetic potential $\mathbf{A}_\nu$ defined on $\mathbb{R}^3_+$ by \eqref{Av},
for which we recall that the associated magnetic field is the constant unit vector that makes an angle $\nu$ with the $x_2x_3$ plane (see \eqref{Bv}).
\begin{defn}\label{def:Energy}
We consider the following reduced Ginzburg-Landau type energy functional
\begin{equation}\label{GLgen}
\mathscr{E}_{\bfr,\nu,\beta,\ell,\ell_3}(\f)=\displaystyle{\int_{\mathcal{D}_{\ell,\ell_3,\beta}}\left(|(-i\nabla+\mathbf{A}_\nu)\f|^2-\bfr|\f|^2+\frac{\bfr}{2}|\f|^4\right)\,\mathrm{d}x},
\end{equation}
for $\f$ in the space
\begin{multline*}\label{spaceS}
\mathcal{S}_{\ell,\ell_3,\beta}=\Big\{\f \in {\color{black}
L^2(\mathcal{D}_{\ell,\ell_3,\beta})}~:~ (-i\nabla+\mathbf{A}_\nu)\f\in L^2(\mathcal{D}_{\ell,\ell_3,\beta}, \C^3),\\
\f=0 \;\; \text{on} \;\; \partial(\mathcal{D}_{\ell,\ell_3,\beta})\backslash \{x_1=0\}
\Big\}.
\end{multline*}
Furthermore, we define
\begin{align}
E(\bfr,\nu,\beta,\ell,\ell_3)= \inf_{\f \in \mathcal{S}_{\ell,\ell_3,\beta}} \mathscr{E}_{\bfr,\nu,\beta,\ell,\ell_3}(\f).
\end{align}
\end{defn}

\begin{rem}\label{rem:existence}
We recall what we mentioned in Remark \ref{rem:existenceInitial}. The existence of the limit
\begin{align}
\lim_{\ell\rightarrow \infty} \frac{1}{4\ell^2} E(\bfr,\nu,\beta=0,\ell,\ell)=e(\bfr,\nu,\beta=0),
\end{align}
was proved in \cite{AS07} and \cite{FKP13} (see in particular \cite[Theorem 3.9]{FKP13} with \cite{FK13}).
Following exactly the same steps of the proof of Theorem 3.9 in \cite{FKP13}, we can actually easily show that for all sequence $\{ (\ell^{(n)}, \ell_3^{(n)}) \}_{n}$ satisfying
\begin{equation}\label{limit-1}
0< \ell^{(n)} \underset{n\rightarrow+\infty}{\longrightarrow} +\infty,\quad\text{and}\quad
C^{-1} < \frac{\ell^{(n)}}{\ell_3^{(n)}} < C \quad \forall n,
\end{equation}
with a fixed constant $C>0$, we have that the following limit
\begin{align}
e(\bfr,\nu,\beta=0):= \lim_{n\rightarrow +\infty} \frac{1}{4\ell^{(n)}\ell_3^{(n)}} E(\bfr,\nu,\beta=0,\ell^{(n)},\ell_3^{(n)})
\end{align}
exists, and is independent of the choice of sequence $\{ (\ell^{(n)}, \ell_3^{(n)}) \}_{n}$.
\end{rem}

We will need the following result which gives in particular the existence of a minimizer and that all the minimizers have a good decay at infinity in the transverse variable $x_1$.
\begin{lem}\label{decaytransverse}
Suppose that $\bfr\in[\Theta_0,1]$ and $\beta\in\left[0,\frac{\pi}{2}\right)$ are fixed given constants. For all $\nu\in\left[0,\frac{\pi}{2}\right]$, $\ell>0$ and $\ell_3>0$, the functional $\mathscr{E}_{\bfr,\nu,\beta,\ell,\ell_3}$ in \eqref{GLgen} has a minimizer, and any minimizer $\underline{\f}=\varphi_{\bfr,\nu,\beta,\ell,\ell_3}$ satisfies
\begin{equation}\label{infinitynorm}
\mathscr{E}_{\bfr,\nu,\beta,\ell,\ell_3}(\underline{\f})=E(\bfr,\nu,\beta,\ell,\ell_3), \quad \Vert \underline{\f}\Vert_{L^\infty(\mathcal{D}_{\ell,\ell_3,\beta})}\le 1.
\end{equation}
and
\begin{equation}\label{eq:GLeq}
\displaystyle{\displaystyle{\int_{\mathcal{D}_{\ell,\ell_3,\beta}} {\color{black}\Big\{ }| (-i\nabla+\mathbf{A}_\nu)\underline{\f}|^2-\bfr|\underline{\f}|^2+\frac{\bfr}{2}|\underline{\f}|^4\,{\color{black}\Big\} }\mathrm{d}x}=-\frac{\bfr}{2}\displaystyle{\int_{\mathcal{D}_{\ell,\ell_3,\beta}}|\underline{\f}|^4}\,\mathrm{d}x}.
\end{equation}

Furthermore, there exists a constant $C(\bfr,\beta)$ such that if $\nu\in\left[0,\frac{\pi}{2}\right]$, $\ell>0$ and $\ell_3>0$, any minimizer $\underline{\f}$ satisfies
\begin{equation}\label{3.14}
\displaystyle{\int_{\mathcal{D}_{\ell,\ell_3,\beta}\cap\{x_1>4\}}\frac{x_1}{(\ln x_1)^2}\Big(|(-i\nabla+\mathbf{A}_\nu)\underline{\f}|^2+|\underline{\f}|^2+x_1^2|\underline{\f}|^4\Big)\,\mathrm{d}x}\le C(\bfr,\beta)\ell\ell_3.
\end{equation}
\end{lem}
\begin{proof}

This statement is proved in \cite[ Theorem 3.6]{FKP13} (see also \cite{P02}) for $\beta=0$ and $\ell=\ell_3$.
The fact that we have a different domain $\mathcal{D}_{\beta,\ell,\ell_3}$ depending on the two additional parameters $\beta$ and $\ell_3$ does not change anything because they are both fixed in the proof.

We omit the details.
\end{proof}
The following result is a key lemma which generalizes Theorem 3.9 in \cite{FKP13}. The proof of this result strongly relies on the generalized result presented in Remark \ref{rem:existence}. We refer to \cite{FKP13} (and \cite{FK13}) for the technical details.
\begin{lem}\label{key}
Let $C>0$ and let $\{ (\ell^{(n)}, \ell_3^{(n)}) \}_{n}$ be a sequence satisfying \eqref{limit-1}.
We have that
\begin{equation*}
e(\bfr,\nu,\beta):=\underset{n\rightarrow+\infty}{\lim} \ \frac{1}{4\ell^{(n)}\ell_3^{(n)}}E(\bfr,\nu,\beta,\ell^{(n)},\ell_3^{(n)})
\end{equation*}
exists, is independent of the choice of sequence $\{ (\ell^{(n)}, \ell_3^{(n)}) \}_{n}$ and moreover is independent of $\beta$.
\end{lem}
Thanks to Lemma \ref{key}, $e(\bfr,\nu,\beta)$ is independent of $\beta$. Hence we can omit $\beta$ from the notation and write
\begin{equation}\label{eq:Limit1bis}
e(\bfr,\nu)= e(\bfr,\nu,\beta)=e(\bfr,\nu,0).
\end{equation}

\begin{proof}[Proof of Lemma \ref{key}]
We fix $C>0$ and choose a sequence $\{ (\ell^{(n)}, \ell_3^{(n)}) \}_{n}$ such that \eqref{limit-1} holds.
For convenience, below we do not systematically indicate the dependence on $n$ in the notation.

For all $\tilde{\ell}>0$, we define the following box
$$\mathcal{D}_{\ell,\ell_3 ,\beta}^{\tilde{\ell} }=\{x=(x_1,x_2,x_3)\in\mathcal{D}_{\ell ,\ell_3 ,\beta}~:~ x_1\in(0,\tilde{\ell} )\},
$$
and for all $\ell>0$, $\ell_3>0$, $\nu\in\left(0,\frac{\pi}{2}\right)$ and $\beta\in\left[0,\frac{\pi}{2}\right)$, we denote by $E^{\mathsf{Dir}}_{\tilde{\ell}}(\bfr,\nu,\beta,\ell,\ell_3)$
the minimal energy of the functional $\mathscr{E}_{\bfr,\nu,\beta,\ell,\ell_3}$
for which the set of test functions is given by all $\varphi\in \mathcal S_{\ell,\ell_3,\beta}$, such that $\varphi$ satisfies  the following Dirichlet condition
$$\varphi=0 \;\; \text{on} \;\; \partial\mathcal{D}_{\ell,\ell_3,\beta}^{\tilde{\ell}}\backslash\{x_1=0\}, \q  \text{and} \; \varphi\ \text{extended by} \ 0 \ \text{outside of $\mathcal D_{\ell,\ell_3,\beta}^{\tilde\ell}$}.
$$
We are going to give upper bound and lower bounds on $E(\bfr,\nu,\beta,\ell ,\ell_3 )$ involving the quantity $E^{\mathsf{Dir}}_{\tilde{\ell}}(\bfr,\nu,\beta,\ell,\ell_3)$.

The upper bound is obvious since, by inclusion of the variational spaces, we have
\begin{equation}\label{Dir1}
E^{\mathsf{Dir}}_{\tilde{\ell} }(\bfr,\nu,\beta,\ell ,\ell_3 )\ge E(\bfr,\nu,\beta,\ell ,\ell_3 ).
\end{equation}

To obtain a lower bound on $E(\bfr,\nu,\beta,\ell,\ell_3 )$, we start by choosing $\tilde{\ell}$ as follows
\begin{equation}\label{eq:varho}
\tilde{\ell} =\ell^{\varrho},
\end{equation}
where $\varrho>0$ is a constant which will be chosen later on. We consider a real number $\tilde{\ell}_{-}$ such that
\begin{equation}\label{cond_tilde}
4\le \tilde{\ell}_{-}<\tilde{\ell} \quad \text{with} \quad |\tilde{\ell}-\tilde{\ell}_{-}|=R(\ell),\end{equation} where $R(\ell)$ is a strictly positive quantity which can depend on $\ell$ and which will be given below. We consider two smooth functions $\chi_1$ and $\chi_2$ defined on $\RRh$ and constituting a partition of unity such that for all $x\in \RRh\supset\mathcal{D}_{\ell,\ell_3,\beta}
$:
\begin{equation*}
\begin{split}
\chi_1(x)=\chi_1(x_1,x_2,x_3) =\left\{\begin{array}{ccl} 1 & \text{if} & x_1\in[0,\tilde{\ell}_{-}],\\
0 & \text{if} & x_1\in[\tilde{\ell}
,+\infty),
\end{array}\right.
\end{split}
\end{equation*}
and
\begin{equation*}
\begin{split}\chi_2(x)=\chi_2(x_1,x_2,x_3) =\left\{\begin{array}{ccc} 0 & \text{if} & x_1\in[0,\tilde{\ell}_{-}],\\
1 & \text{if} & x_1\in[\tilde{\ell}
,+\infty),
\end{array}\right.
\end{split}
\end{equation*}
with
\begin{equation*}
\chi_{1}^2+\chi_{2}^2=1 \quad \text{on} \ \RRh,
\end{equation*}
and
\begin{equation}\label{cutoff_aux}
\underset{x\in\RRh}{\sup}|\nabla\chi_{j}(x)|\le \frac{C}{R(\ell)}, \quad \text{for all} \ j\in \{1,2\}.
\end{equation}
For all $\varphi\in \mathcal{S}_{\ell,\ell_3,\beta}$, using the fact that $$\displaystyle{\int_{\mathcal{D}_{\ell,\ell_3,\beta}}(\chi_j^2(x)-\chi_j^4(x))|\varphi(x)|^4\,\mathrm{d}x}\ge 0
$$
for all $j\in\{1,2\}$ (since $0\le\chi_j(x)\le 1$), we have the following standard decomposition formula (see for instance \cite[Section 5.1]{FKP13})
\begin{multline}\label{aux_lower}
\mathscr{E}_{\bfr,\nu,\beta,\ell,\ell_3}(\varphi)\ge \mathscr{E}_{\bfr,\nu,\beta,\ell,\ell_3}(\chi_1\varphi)+\mathscr{E}_{\bfr,\nu,\beta,\ell,\ell_3}(\chi_2\varphi)\\-\Vert \varphi|\nabla\chi_1|\Vert^2_{L^2(\RRh)}-\Vert \varphi|\nabla\chi_2|\Vert^2_{L^2(\RRh)}.
\end{multline}
To obtain a lower bound to $E(\bfr,\nu,\beta,\ell ,\ell_3 )$, we are going to use the \eqref{aux_lower}  by giving an estimate on each term of the right hand side of the inequality.

On one hand we have, by definition 
and using the property of $\chi_1$, that for all $\f\in \mathcal{S}_{\ell,\ell_3,\beta}$
\begin{equation}\label{mainlower}
\displaystyle{\int_{\mathcal{D}_{\ell,\ell_3,\beta}
}\left(|(-i\nabla+\mathbf{A}_\nu)\chi_1\f|^2-\bfr|\chi_1\f|^2+\frac{\bfr}{2}|\chi_1\f|^4\right)\,\mathrm{d}x}\ge E^{\mathsf{Dir}}_{\tilde{\ell} }(\bfr,\nu,\beta,\ell ,\ell_3).
\end{equation}
On the other hand, using Proposition~\ref{prop:NoBoundary} (where $\chi_2\f$ is extended by zero to a regular function on all of $\R^3_+\supset\mathcal{D}_{\ell,\ell_3,\beta}
$), we have
\begin{align}\label{positif_lower}
\int_{\mathcal{D}_{\ell,\ell_3,\beta}
}\left(|(-i\nabla+\mathbf{A}_\nu)\chi_2\f|^2-\bfr|\chi_2\f|^2+\frac{\bfr}{2}|\chi_2\f|^4\right)\,\mathrm{d}x \geq 0,
\end{align}
since $\bfr \leq 1$.

It remains to give an estimate on the remainders of \eqref{aux_lower}. Considering the support of the functions $|\nabla\chi_j|$ ($j\in\{1,2\}$), and using \eqref{cutoff_aux}, we have the following inequality
\begin{equation}\label{Restemin_aux}
\underset{j\in\{1,2\} }{\sum}\Vert \varphi|\nabla\chi_j|\Vert^2_{L^2(\RRh)}\le \frac{C^2}{R(\ell)^2}\displaystyle{\int_{\tilde{\ell}_-\le x_1\le \tilde{\ell}}|\f|^2\,\mathrm{d}x}.
\end{equation}
To estimate $\displaystyle{\int_{\tilde{\ell}_-\le x_1\le \tilde{\ell}}|\f|^2\,\mathrm{d}x}$, we use Lemma \ref{decaytransverse} which gives in particular that for any
$\gamma\in[1,2)$ there exists a constant $C(\bfr,\beta,\gamma)>0$ (which only depends on $\bfr$, $\beta$ and $\gamma$) such that:
\begin{equation}\label{decay_aux}
\displaystyle{\int_{\mathcal{D}_{\ell,\ell_3,\beta}\cap\{x_1>4\}}x_1^\gamma|\underline{\f}|^2\,\mathrm{d}x}\le C(\bfr,\beta,\gamma)\ell\ell_3.
\end{equation}
Using \eqref{decay_aux} for $\gamma$ in the form $\gamma=1-\epsilon$ with $\epsilon\in(0,1)$, we have that there exist a constant $C(\bfr,\beta)>0$ (also depending the choice of $\epsilon$)
such that
\begin{equation}\label{reste_aux}
\displaystyle{\int_{\tilde{\ell}_-\le x_1\le \tilde{\ell}}|\f|^2\,\mathrm{d}x}\le C(\bfr,\beta)\frac{\ell\ell_3}{\ell^{\varrho(1-\epsilon)}}.
\end{equation}
We choose
\begin{equation}\label{eq:tildel2new}
0<R(\ell)=R<2 \quad \text{and} \quad \varrho=\frac{1}{2},
\end{equation}
where $R$ is a fixed constant independent of $\ell$ and where we recall that $\varrho$ is introduced in \eqref{eq:varho}. Thus, we are lead to consider
\begin{equation}\label{eq:tildel2}
\tilde{\ell} =\ell^{1/2}.
\end{equation}
With the choices given by \eqref{eq:tildel2new}, we have $\tilde{\ell}_-=\tilde{\ell}+o(\tilde{\ell})=\ell^{1/2}+o(\ell^{1/2})$. Therefore, we have with \eqref{reste_aux} that
\begin{equation}\label{Restemin_aux_bis}
\displaystyle{\int_{\tilde{\ell}_-\le x_1\le \tilde{\ell}}|\f|^2\,\mathrm{d}x}\le C(\bfr,\beta)\ell^{1/2+\epsilon/2}\ell_3,
\end{equation}
which finally gives
\begin{equation}\label{Restemin}
\underset{j\in\{1,2\} }{\sum}\Vert \varphi|\nabla\chi_j|\Vert^2_{L^2(\RRh)}\le C(\bfr,\beta)\ell^{1/2-\epsilon/2}\ell_3,
\end{equation}
Using \eqref{aux_lower} with a minimizer
$\underline{\f}=\f_{\bfr,\nu,\beta,\ell,\ell_3}$ of the functional $\mathscr{E}_{\bfr,\nu,\beta,\ell,\ell_3}$, and the estimates \eqref{mainlower}, \eqref{positif_lower} and \eqref{Restemin}, we obtain \begin{equation}\label{Dir2}
E^{\mathsf{Dir}}_{\tilde{\ell} }(\bfr,\nu,\beta,\ell ,\ell_3 )+o(\ell\ell_3)\le E(\bfr,\nu,\beta,\ell ,\ell_3 ).
\end{equation}
Relying on \eqref{Dir1} and \eqref{Dir2}, we can deal with $E(\bfr,\nu,\beta,\ell ,\ell_3 )$ using $E^{\mathsf{Dir}}_{\tilde{\ell} }(\bfr,\nu,\beta,\ell ,\ell_3 )$ up to a remainder in $o(\ell\ell_3 )$.
We denote by $\ell_*$ the length of the segment joining the two points $(x_1=\tilde{\ell},x_2=\ell,x_3=0)$ and $(x_1=\tilde{\ell},x_2=\ell+\tilde{\ell}\tan\beta,x_3=0)$.
The value of $\ell_*$ is given by
\begin{equation}\label{eq:ellstar}
\ell_* =\left(\tan\beta\right)\tilde{\ell} =\left(\tan\beta \right)\ell ^{1/2}.
\end{equation}
For $\underline{\ell} =\ell -\ell_* $ and $\overline{\ell}=\ell+\ell_*$, we have the following geometrical inclusions
\begin{equation}\label{inclu}
\mathcal{D}_{\underline{\ell},\ell_3,0}^{\tilde{\ell}}\subset \mathcal{D}_{\ell,\ell_3,\beta}^{\tilde{\ell}} \subset \mathcal{D}_{\overline{\ell},\ell_3,0}^{\tilde{\ell}}.
\end{equation}
Using the inclusions \eqref{inclu}, we have the following inequalities:
\begin{equation}\label{eq:in1}
\frac{1}{4\ell\ell_3}E^{\mathsf{Dir}}_{\tilde{\ell}}(\bfr,\nu,\beta,\ell,\ell_3)\le \frac{\underline{\ell}\ell_3}{\ell\ell_3}\left(\frac{1}{4\underline{\ell}\ell_3}E^{\mathsf{Dir}}_{\tilde{\ell}}\left(\bfr,\nu,0,\underline{\ell},\ell_3\right)\right),
\end{equation}
and
\begin{equation}\label{eq:in2}
\frac{\overline{\ell}\ell_3}{\ell\ell_3}\left(\frac{1}{4\overline{\ell}\ell_3}E^{\mathsf{Dir}}_{\tilde{\ell}}\left(\bfr,\nu,0,\overline{\ell},\ell_3\right)\right)\le \frac{1}{4\ell\ell_3}E^{\mathsf{Dir}}_{\tilde{\ell}}(\bfr,\nu,\beta,\ell,\ell_3).
\end{equation}
Therefore, using \eqref{Dir1}, \eqref{Dir2}, we can get ride of the Dirichlet energy and obtain from \eqref{eq:in1} and \eqref{eq:in2} the following inequalities:
\begin{equation}\label{eq:in1bis}
\frac{1}{4\ell\ell_3}E(\bfr,\nu,\beta,\ell,\ell_3)\le \frac{\underline{\ell}\ell_3}{\ell\ell_3}\left(\frac{1}{4\underline{\ell}\ell_3}\left(E\left(\bfr,\nu,0,\underline{\ell},\ell_3\right)+o(\ell\ell_3)\right)\right),
\end{equation}
and
\begin{equation}\label{eq:in2bis}
\frac{\overline{\ell}\ell_3}{\ell\ell_3}\left(\frac{1}{4\overline{\ell}\ell_3}E\left(\bfr,\nu,0,\overline{\ell},\ell_3\right)\right)\le \frac{1}{4\ell\ell_3}\left(E(\bfr,\nu,\beta,\ell,\ell_3)+o(\ell\ell_3)\right).
\end{equation}
As $\underline{\ell}=\ell+o(\ell)$ and $\overline{\ell}=\ell+o(\ell)$ (see \eqref{eq:ellstar}), we finally obtain (using the result given in Remark \eqref{rem:existence}) that
\begin{equation}\label{aim1upper}
\underset{n\rightarrow+\infty}{\lim\sup}\ \frac{1}{4\ell^{(n)}\ell_3^{(n)}}E(\bfr,\nu,\beta,\ell^{(n)},\ell_3^{(n)})\le\underset{n\rightarrow+\infty}{\lim}\frac{1}{4\ell^{(n)}\ell_3^{(n)}}E\left(\bfr,\nu,0,\ell^{(n)},\ell_3^{(n)}\right),
\end{equation}
and
\begin{equation}\label{aim2upper}
\underset{n\rightarrow+\infty}{\lim}\frac{1}{4\ell^{(n)}\ell_3^{(n)}}E\left(\bfr,\nu,0,\ell^{(n)},\ell_3^{(n)}\right)\le\underset{n\rightarrow+\infty}{\lim\inf}\frac{1}{4\ell^{(n)}\ell_3^{(n)}}E(\bfr,\nu,\beta,\ell^{(n)},\ell_3^{(n)}).
\end{equation}
The result is obtained using \eqref{aim1upper} and \eqref{aim2upper}.
\end{proof}

\subsection{Proof of Theorem \ref{thm:mon}}
\label{Sec:NewCoords}
In this subsection we will give the proof of Theorem \ref{thm:mon} by comparing to a problem in convenient coordinates.

For $L, L_3>0$ and $\alpha \in (0,\pi)$ we define
\begin{equation*}
\widetilde{\mathcal{D}}=\widetilde{\mathcal{D}}_{L,L_3,\alpha}=\widetilde{\mathcal{D}}_{L,\alpha}\times(-L_3,L_3),
\end{equation*}
with
\begin{align}
\label{tildeDnu}
\widetilde{\mathcal{D}}_{L,\alpha}=\left\{(v_1,v_2)\in\R^2~:~ v_1>-v_2,\;\; \left| (\tan\alpha) v_1- v_2 \right| \leq \frac{L}{\sqrt{2}} (1+ \tan \alpha) \right\}.
\end{align}
\begin{rem}\label{rem:BdryArea}
Notice that the area of $\widetilde{\mathcal{D}}_{L,L_3,\alpha} \cap \{ v_1=-v_2\}$ is $4 L L_3$.
\end{rem}

Define furthermore the functional
\begin{align}\label{newGLH}
&\EH(\THfder)=\EH_{\bfr,\nu,\alpha,L,L_3}(\THfder)  \\
&=\int_{\widetilde{\mathcal{D}}_{L,L_3,\alpha}}
|\partial_1\THfder|^2+\TT|\partial_2\THfder|^2+|(-i\partial_3+v_1)\THfder|^2
 \displaystyle{-\bfr|\THfder|^2+\frac{\bfr}{2}\T|\THfder|^4\,
 \mathrm{d}v_1\mathrm{d}v_2\mathrm{d}v_3},\nonumber
\end{align}
for $\tilde{\f}$ in the space
\begin{multline*}\label{spaceSH}
\sH_{L,L_3, \alpha}=\left\{\THfder \in L^2(\widetilde{\mathcal{D}}_{L,L_3,\alpha})~:~ \right.\\
\qq\qq\qq \left. \left(|\partial_1\THfder|^2+\TT|\partial_2\THfder|^2+|(-i\partial_3+v_1)\THfder|^2\right)\in L^2(\widetilde{\mathcal{D}}_{L,L_3,\alpha}),\right. \\
\left.\THfder=0 \;\; \text{on} \;\; \partial\widetilde{\mathcal{D}}_{L,L_3,\alpha}\backslash \{v_2=-v_1\}\right\}.
\end{multline*}
Also, introduce the following ground state energy,
\begin{equation}\label{newminH}
\widetilde{E}=\widetilde{E}(\bfr,\nu,\alpha, L, L_3)=\underset{\THfder\in\sH_{L,L_3, \alpha}}{\inf}\EH(\THfder).
\end{equation}
In the following lemma, we collect the needed results about $\EH$.
\begin{lem}\label{l1}
Suppose that $\bfr\in(\Theta_0,1]$ and $\alpha\in\left(0,\pi\right)$. For all $\nu\in\left(0,\frac{\pi}{2}\right)$, $L>0$ and $L_3>0$, the functional $\EH_{\bfr,\nu,\alpha,L,L_3}$ defined in \eqref{newGLH} has a minimizer, and any minimizer  $\tilde{\f}=\f_{\bfr,\nu,\alpha,L,L_3}$ satisfies \begin{equation}
\EH_{\bfr,\nu,\alpha,L,L_3}(\tilde{\f})=\widetilde{E}(\bfr,\nu,\alpha, L, L_3), \quad \Vert \tilde{\f}\Vert_{L^\infty(\widetilde{\mathcal{D}}_{L,L_3,\alpha})}\le 1.
\end{equation}
and the following Euler-Lagrange equation associated with the minimization problem \eqref{newminH}
\begin{equation}\label{newGLeq}
\Big(-\partial_{v_1}^2-\TT\partial_{v_2}^2+(-i\partial_{v_3}+v_1)^2\Big)\THfder=\bfr(1-\T|\THfder|^2)\THfder,
\end{equation}
and
\begin{align}\label{id}
\displaystyle{\int_{\widetilde{\mathcal{D}}_{L,L_3,\alpha}}
\Big\{ |\partial_1\THfder|^2+\TT|\partial_2\THfder|^2+|(-i\partial_3+v_1)\THfder|^2-\bfr|\THfder|^2+\frac{\bfr}{2}\T|\THfder|^4\,\Big\} \mathrm{d}v} \nonumber \\
=-\frac{\bfr}{2}\T\displaystyle{\int_{\widetilde{\mathcal{D}}_{L,L_3,\alpha}}|\THfder|^4\,\mathrm{d}v}.
\end{align}
Furthermore, we have the following conclusions,
\begin{enumerate}[label=(\roman*)]
\item\label{1} In the case when \begin{equation}\label{para1}
\alpha = \arctan(\tan^2(\nu)),\quad
L_3 = \ell,\quad
L = \sqrt{2} \ell \sin(\nu),
\end{equation} we have the identity
\begin{align}
\label{eq:Energies_ChgCoord}
\widetilde{E}(\bfr,\nu,\alpha, L, L_3) = E(\bfr,\nu,\ell).
\end{align}
In particular, still with this special relation between the parameters,
\begin{align}\label{eq:NormalizedEnergies}
\sqrt{2} \sin(\nu) \frac{\widetilde{E}(\bfr,\nu,\alpha, L, L_3)}{4 L L_3} =  \frac{E(\bfr,\nu,\ell)}{4 \ell^2}.
\end{align}
\item \label{2} Let $C>0$ and let $\{ (L^{(n)}, L_3^{(n)}) \}_{n}$ be a sequence satisfying
\begin{equation}\label{limit-1new}
0< L^{(n)} \underset{n\rightarrow+\infty}{\longrightarrow} +\infty,\quad\text{and}\quad
C^{-1} < \frac{L^{(n)}}{ L^{(n)}_3} < C \quad \forall n,
\end{equation}
with a fixed constant $C>0$. The limit
\begin{equation*}
\tilde{e}=\tilde{e}(\bfr,\nu,\alpha)=\underset{n\rightarrow+\infty}{\lim}\frac{1}{4L^{(n)} L_3^{(n)} }\widetilde{E}(\bfr,\nu,\alpha, L^{(n)}, L_3^{(n)})
\end{equation*}
is finite and is independent of $C$ and of the sequence $\{ (L^{(n)}, L_3^{(n)}) \}_{n}$.
\item \label{3} The quantity $\tilde{e}(\bfr,\nu,\alpha)$ is independent of $\alpha$.
\end{enumerate}
\end{lem}
\begin{proof}
Obtaining the Euler-Lagrange equation \eqref{newGLeq} is rather standard. We use \eqref{newGLeq} to obtain \eqref{id} by an integration by parts after multiplying by the conjugate of $\THfder$.

The proof of the \ref{1}
is by a change of coordinates.
Composing the two changes of variables
\begin{equation*}
\left\{\begin{array}{ccl}
x_1 & = & -u_1\cv -u_2\sv
\\ x_2 & = & u_1\sv -u_2\cv
\\ x_3 & = & u_3
\end{array}\right. \qquad \text{and next} \qquad
\left\{\begin{array}{ccl}
u_1 & = & -v_1
\\ u_2 & = & \frac{v_2}{-\T}
\\ u_3 & = & v_3
\end{array},\right.
\end{equation*}
which are respectively a rotation and a dilatation in the second variable, changes the functional expression and maps the domain $\mathcal{D}_\ell$ onto the domain $\widetilde{\mathcal{D}}_{\ell,\nu}$, where $\widetilde{\mathcal{D}}_{\ell,\nu}$ is given by
\begin{equation*}
\widetilde{\mathcal{D}}_{\ell,\nu}=\left\{(v_1,v_2)\in\R^2\,:\, v_1>-v_2,\; -\frac{\ell\T}{\cv}< v_2-v_1\TT<\frac{\ell\T}{\cv}\right\}\times(-\ell, \ell).
\end{equation*}
Combined with the change of function $\f=\sqrt{\T}\,\THfder$ we get
$$
\EH(\THfder) = \mathscr{E}_{\bfr,\nu,\ell}(\f).
$$
It is also clear that when $\alpha, L, L_3$ satisfy \eqref{para1} we have
$$
\widetilde{\mathcal{D}}_{L,L_3,\alpha} = \widetilde{\mathcal{D}}_{\ell,\nu}.
$$
Thus, \eqref{eq:Energies_ChgCoord} is proved. Furthermore, \eqref{eq:NormalizedEnergies} follows immediately using Remark~\ref{rem:BdryArea}.

\smallbreak The second and the third assertions follow from the results of \cite{FKP13}, as mentioned in Remark \ref{rem:existence}.
\end{proof}

\begin{proof}[Proof of Theorem \ref{thm:mon}]
It suffices to prove the monotonicity of $e(\bfr,\nu)$ in $\nu$ restricted to $\nu \in (0,\frac{\pi}{2})$ since $e(\bfr,\cdot)$ is continuous by \cite[Theorem 3.13]{FKP13}. So in the remainder of the proof we work under this restriction.

By Lemma~\ref{l1}, we can define
\begin{equation*}
\widetilde{E}(\bfr,\nu,L)=\widetilde{E}(\bfr,\nu,\alpha=\frac{\pi}{4},L,L),
\end{equation*}
and (as in \eqref{eq:Limit1bis}) \begin{equation}\label{eq:Limit1bisnew}
\tilde{e}(\bfr,\nu)=\tilde{e}(\bfr,\nu,\alpha=\frac{\pi}{4}).
\end{equation}
Using the first point \ref{1} of Lemma~\ref{l1}, we have the following correspondence
\begin{equation}\label{eq:corr}
e(\bfr,\nu)=\sqrt{2}\sv\,\tilde{e}(\bfr,\nu).
\end{equation}
We consider
$$\Delta_{\bfr,\nu}(\varepsilon)=e(\bfr,\nu+\varepsilon)-e(\bfr,\nu).$$
Using \eqref{eq:corr} and Lemma~\ref{l1}, we can write
\begin{equation}\label{delta}
\begin{split}
\Delta_{\bfr,\nu}(\varepsilon)
=\sqrt{2}\underset{L\rightarrow+\infty}{\lim}\left(\sin(\nu+\varepsilon)\frac{\widetilde{E}(\bfr,\nu+\varepsilon,L
)}{4L^2}-\sin(\nu)\frac{\widetilde{E}(\bfr,\nu,L)}{4L^2}\right).
\end{split}
\end{equation}
We take $\varepsilon>0$ and we are looking for a positive lower bound on $\Delta_{\bfr,\nu}(\varepsilon)$ in order to prove the monotonicity. We will use a minimizer of the functional $$\widetilde{\mathscr{E}}_{\bfr,\nu+\varepsilon,L}=\widetilde{\mathscr{E}}_{\bfr,\nu+\varepsilon,\alpha=\frac{\pi}{4},L,L}.$$ This minimizer exists and will be denoted
$\fmine=\tilde{\f}_{\bfr,\nu+\varepsilon,L}$. Therefore we have
\begin{equation*}
\widetilde{\mathscr{E}}_{\bfr,\nu,L}(\fmine)\ge \widetilde{E}(\bfr,\nu,L) \quad \text{and}\quad \widetilde{\mathscr{E}}_{\bfr,\nu+\varepsilon,L}(\fmine)=\widetilde{E}(\bfr,\nu+\varepsilon,L).
\end{equation*}
Upon inserting this in \eqref{delta} we get
$$\Delta_{\bfr,\nu}(\varepsilon)\ge\frac{\sqrt{2}}{4}\underset{L\rightarrow+\infty}{\lim}\frac{1}{L^2}\left(\sin(\nu+\varepsilon)\widetilde{\mathscr{E}}_{\bfr,\nu+\varepsilon,L}(\fmine)-\sin(\nu)\widetilde{\mathscr{E}}_{\bfr,\nu,L}(\fmine)\right).$$
Considering the integral expression given by \eqref{newGLH} and writing
\begin{equation}
\aligned
&\sin(\nu+\varepsilon)\widetilde{\mathscr{E}}_{\bfr,\nu+\varepsilon,L}(\fmine)-\sin(\nu)\widetilde{\mathscr{E}}_{\bfr,\nu,L}(\fmine)\\
=&\sin(\nu+\varepsilon)\widetilde{\mathscr{E}}_{\bfr,\nu+\varepsilon,L}(\fmine)-\sin(\nu)\widetilde{\mathscr{E}}_{\bfr,\nu,L}(\fmine)-\sin(\nu)\widetilde{\mathscr{E}}_{\bfr,\nu+\varepsilon,L}(\fmine)+\sin(\nu)\widetilde{\mathscr{E}}_{\bfr,\nu+\varepsilon,L}(\fmine),
\endaligned
\eeq
we have
\begin{equation*}
\begin{split}
\Delta_{\bfr,\nu}(\varepsilon)\ge \frac{\sqrt{2}}{4}\underset{L\rightarrow+\infty}{\lim}\Big(&\frac{1}{L^2}(\sin(\nu+\varepsilon)-\sin(\nu))\widetilde{E}(\bfr,\nu+\varepsilon,L)
\\&+\frac{1}{L^2}\sin(\nu)(\tan^2(\nu+\varepsilon)-\tan^2(\nu))\displaystyle{
\int_{\widetilde{\mathcal{D}}_{L,L,\frac{\pi}{4}}}|\partial_2\fmine|^2\,\mathrm{d}v}
\\&+\frac{1}{L^2}\sin(\nu)\frac{\bfr}{2}(\tan(\nu+\varepsilon)-\tan(\nu))\displaystyle{\int_{
\widetilde{\mathcal{D}}_{L,L,\frac{\pi}{4}}}|\fmine|^4\,\mathrm{d}v}
\big).
\end{split}
\end{equation*}
For $\varepsilon\ge 0$ small enough, we have that $\tan^2(\nu+\varepsilon)-\tan^2(\nu)\ge0$, so that the term
$$\underset{L\rightarrow+\infty}{\lim}\frac{1}{4L^2}\displaystyle{\int_{\widetilde{\mathcal{D}}_{L,L,\frac{\pi}{4}}}\sin(\nu)(\tan^2(\nu+\varepsilon)-\tan^2(\nu))|\partial_2\fmine|^2\,\mathrm{d}v}$$
is positive and we can discard it in the lower bound. Using the identity \eqref{id} of Lemma \ref{l1} we have
$$\widetilde{E}(\bfr,\nu+\varepsilon,L)=-\frac{\bfr}{2}\tan(\nu+\varepsilon)\displaystyle{\int_{\widetilde{\mathcal{D}}_{L,L,\frac{\pi}{4}}}|\fmine|^4\,\mathrm{d}v}.$$
Therefore, we get
\begin{multline*}
\Delta_{\bfr,\nu}(\varepsilon)\ge \frac{\sqrt{2}}{4}\frac{\bfr}{2}\Big(
\sin(\nu)(\tan(\nu+\varepsilon)-\tan(\nu))
\\-(\sin(\nu+\varepsilon)
-\sin(\nu))\tan(\nu+\varepsilon)\Big)
\underset{L\rightarrow+\infty}{\lim}\frac{1}{L^2}\displaystyle{\int_{\widetilde{\mathcal{D}}_{L,L,\frac{\pi}{4}}}|\fmine|^4\,\mathrm{d}v}.
\end{multline*}
To conclude, as $\nu\in(0,\frac{\pi}{2})$, it is easy to see that there exists $\varepsilon_0>0$ small enough such that for all $0<\varepsilon<\varepsilon_0$
$$-\T(\sin(\nu+\varepsilon)-\sin(\nu))+\sin(\nu+\varepsilon)(\tan(\nu+\varepsilon)-\T)\ge 0.$$
Indeed, by differentiation, we have as $\varepsilon\to 0$
\begin{align*}
-\T(\sin(\nu+\varepsilon)&-\sin(\nu))+\sin(\nu+\varepsilon)(\tan(\nu+\varepsilon)-\T)\\ &= \varepsilon\left(-\T \cos(\nu)+ \sin(v)(1+\tan^2\nu)\right)+O(\varepsilon^2) \\
&=\varepsilon \sv\TT + O(\varepsilon^2).
\end{align*}
This finishes the proof.
\end{proof}

\section{Abrikosov structure on a surface in $3$D}\label{Sec:Abri}

In this section, we prove Theorem \ref{bislem:3.1Par} stated in Section \ref{Sec:statement} below. We are interested in constructing bounded solutions with lattice structure, i.e.
for which the physical quantities -- the density of Cooper pairs $|\psi|^2$, the magnetic field $\mathbf{B}$ and the magnetic current $\Re\left(\overline{\psi} (-i\nabla + \mathbf{A})\psi\right)$ are periodic.
This corresponds to states $\psi$ satisfying the magnetic periodic conditions given by \eqref{3.4bis} below.

Before stating the main result of this section, we should introduce some notations.

First of all, we introduce for an interval $J \subset {\mathbb R}$ 
\begin{align}
L^2_{\rm comp}(J\times \R^2) := \{ \psi \in L^2_{\rm loc}(J \times \R^2)\,:\, \psi \in L^2(J \times K) \text{ for all compact } K \subset {\mathbb R}^2\}.
\end{align}

\subsection{Spectral problem with periodic condition on a parallelogram cell}\label{ParCell}

Let $\mathbf{e}_2$ and $\mathbf{e}_3$ denote the standard unit vectors $(0,1,0)$ and $(0,0,1)$ associated with the variables $x_2$ and $x_3$ (respectively).

\begin{nota}\label{nota-D} Let $R>0$, $R'>0$ and $\theta\in (0,\pi)$. We denote by $\LLat_{R,R',\theta}$ the lattice (in the $x_2x_3$-plane)
 defined by
$$\LLat_{R,R',\theta}=\text{span}_{\Z}\left\{s:=R\mathbf{e}_2, t:=R'\mathbf{e}\right\},
$$
where $\mathbf{e}=\cos\theta \mathbf{e}_2+\sin\theta\mathbf{e}_3$.

We will denote by $\Lat_{R,R',\theta}$ the following fundamental domain
$$\Lat_{R,R',\theta}=\left\{(x_2,x_3)\in \R^2, \frac{x_3}{\tan\theta} \le x_2\le \frac{x_3}{\tan\theta} +R, 0\le x_3\le R'\sin\theta \right\}.$$

\end{nota}

Notice that for convenience of notation, we have assumed that the system of coordinates is chosen so that one of the `legs' of the lattice is parallel to a coordinate axis. In general, there can be an angle $\tau$ between the lattice and the (projection onto the $x_2x_3$-plane of the) magnetic field. Therefore, this choice necessitates a change of convention for the magnetic field compared to what has been used in the first part of the article. As a consequence, we introduce the following notation for the magnetic vector potential and field,

\begin{equation}\label{Avgen}
\mathbf{A}=\mathbf{A}_{\nu,\ang}=\left(\begin{array}{c} 0\\ x_1\cv \sin\ang \textcolor{black}{-\frac{1}{2}x_3\sin\nu}\\ -x_1 \cv\cos\ang+\textcolor{black}{\frac{1}{2}}x_2 \sv\end{array}\right),
\end{equation}
and
\begin{equation}\label{Bvgen}
\mathbf{B}=\mathbf{B}_{\nu,\ang}= \nabla \times \mathbf{A}_{\nu,\ang}=\left(\begin{array}{c} \sv\\ \cv \cos\ang \\ \cv \sin\ang\end{array}\right).
\end{equation}

The magnetic flux $\Phi_{\mathbf{A}_{\nu,\ang}}$ over the cell $\Lat_{R,R',\theta}$ is 
\begin{equation}\label{fluxbis}
\Phi_{\mathbf{A}_{\nu,\ang}}=\frac{1}{2\pi}\displaystyle{\iint_{\Lat_{R,R',\theta}}\mathbf{B}_{\nu,\ang}\cdot{\mathbf e}_1\,\mathrm{d}x_2\mathrm{d}x_3}=\frac{RR'\sv\sin\theta}{2\pi}.
\end{equation}

The flux $\Phi_{\mathbf{A}_{\nu,\ang}}$ is obviously proportional to the area of the cell $RR'\sin\theta$.
As we are interested in lattice states (i.e. satisfying a double-periodic condition over the lattice $\LLat_{R,R',\theta}$),
it is natural to impose the following flux quantization condition
\begin{equation}\label{quantbis}
RR'\sv\sin\theta\in 2\pi\Z.
\end{equation}

Under the flux quantization condition \eqref{quantbis}, we can introduce the following magnetic periodicity condition on wave functions:
\begin{equation}
\label{3.4bis}
\left\{
\begin{aligned}
&\psi(x_1,x_2+R,x_3)=\psi(x)e^{\qe i\textcolor{black}{\frac{R}{2}} x_3\sv},
\\& \psi(x_1,x_2+R'\cos\theta,x_3+R'\sin\theta)=\psi(x)e^{\qe i\textcolor{black}{\frac{R'}{2}}\sin\nu(x_3\cos\theta- \textcolor{black}{x_2\sin\theta})}.
\end{aligned}\right.
\end{equation}
The magnetic periodic conditions \eqref{3.4bis} can more compactly be written as 
\begin{equation}\label{eq:per_gen}
\psi(x+\latti)=\psi(x)e^{-ig_\latti(x)}, \quad \forall \latti \in \LLat_{R,R',\theta},
\end{equation}
where $g_\latti$ is the function defined as follows
\begin{equation}\label{eq:autre2}
g_\latti(x)=g_{\nn s+\mm t}(x)
=\frac{b}{2}\left((\nn s+\mm t)\wedge x +(\nn s)\wedge (\mm t)\right),
\end{equation} 
for all $\latti=\nn s+\mm t$ in $\LLat_{R,R',\theta}$ ($\nn ,\mm \in\Z$), with $b=\sin\nu$. 
Here, for $x,y \in {\mathbb R}^3$, we have used the definition
$$
x\wedge y = x_2 y_3 - x_3 y_2.
$$

Notice that because 
$$\psi(x+\nn s+\mm t)=\psi((x+\nn s)+\mm t)=\psi((x+\mm t)+\nn s),$$ when $\psi(x)\neq 0$ we are lead to the condition 
\begin{equation}\label{eq:onlythisautre}
e^{-ig_{\nn s+\mm t}(x)}=e^{-ig_{\nn s}(x+\mm t)}e^{-ig_{\mm t}(x)}\left(=e^{-ig_{\mm t}(x+\nn s)}e^{-ig_{\nn s}(x)}\right).
\end{equation}
With our choice of $g_\latti$ given by \eqref{eq:autre2}, \eqref{eq:onlythisautre} follows from \eqref{quantbis}. Thus the quantization of the flux is important for consistency of these periodic conditions. 

We also state for later reference the following identity for wave functions $\psi$ satisfying \eqref{3.4bis} for all $w \in \LLat_{R,R',\theta}$,
\begin{equation}\label{gradidbis0}
\left.\NAAT\psi\right\vert_{x+w}=e^{-ig_w(x)}\left.\NAAT\psi\right\vert_{x}.
\end{equation}

Under condition \eqref{quantbis}, we consider the following eigenvalue problem on the set $\mathbb{R}_+\times \Lat_{R,R',\theta}$
\begin{equation}\label{3.5gen}
\left\{\begin{array}{ll}
(-i\nabla\qeb \mathbf{A}_{\nu,\ang})^2\psi=\lambda\psi & \ \text{in} \ \mathbb{R}_+\times \DCper\, ,\\
\frac{\partial\psi}{\partial x_1}=0 & \ \text{on} \ \p{\mathbb R}^3_+\, ,\\
\psi \ \text{satisfying} \ \eqref{3.4bis}\, . &
\end{array}\right.
\end{equation}

\begin{lem}\label{lem:spectralPB}
We can associate with the spectral problem \eqref{3.5gen}, a unique self-adjoint operator which corresponds to the Friedrichs extension of the following quadratic form
\begin{equation*}
\QQvpT(\psi)=\displaystyle{\int_{\R_+\times \Lat_{R,R',\theta}}|(-i\nabla\qeb \AAT)\psi|^2\,\mathrm{d}x},
\end{equation*}
defined for all $\psi$ in the form domain \begin{multline}\label{eq:Hvt}
\HCper=\{\psi\in L^2_{\rm comp}(\mathbb{R}^3_+)~:~ \\
(-i\nabla\qeb \mathbf{A}_{\nu,\ang})\psi\in L^2_{\rm comp}(\mathbb{R}^3_+,\mathbb{C}^3),\;\; \psi \ \text{satisfies} \ \eqref{3.4bis}\}.
\end{multline}
\end{lem}
\begin{proof}
This is a standard result since $\QQvpT$ defines a closed quadratic form on $\HCper$.
\end{proof}

\begin{rem}
Notice that if $\varphi, \psi$ satisfy \eqref{3.4bis}, $\partial_{x_1} \psi \big|_{x_1=0} = 0$ and are sufficiently regular, then
\begin{align}
\QQvpT(\varphi,\psi) = \langle \varphi, (-i\nabla + \mathbf{A}_{\nu,\ang})^2\psi\rangle.
\end{align}
This follows by integration by parts using that (by \eqref{eq:per_gen} and \eqref{gradidbis0}),
\begin{align*}
\overline{\varphi} (-i\nabla + \mathbf{A}_{\nu,\ang})\psi\big|_{x+w} &= (\overline{e^{-ig_w(x)} \varphi(x)} e^{-ig_w(x)} \left((-i\nabla + \mathbf{A}_{\nu,\ang})\psi\big|_{x}\right) \\
&= \overline{\varphi(x)} (-i\nabla + \mathbf{A}_{\nu,\ang})\psi(x),
\end{align*}
causing the boundary terms to cancel each other pairwise.
\end{rem}

\begin{nota}\label{zetavRbeforeGEN}
We denote by $\mathcal{L}_{\ang}^{\mathsf{per}}(\nu,\theta)$ the linear self-adjoint operator associated with the eigenvalue problem \eqref{3.5gen}, and by $\zeta_\ang(\nu,R,R',\theta)$ the following quantity
\begin{equation*}
\zeta_\ang(\nu,R,R',\theta)=\inf\spe\left(\mathcal{L}_{\ang}^{\mathsf{per}}(\nu,\theta)\right).
\end{equation*}
\end{nota}

\subsection{Statement of the result}\label{Sec:statement}
We present here the statement of the result we prove in this section. Similarly to \eqref{GL} and \eqref{GLgen}, we define (with the space $\HCper$ defined in \eqref{eq:Hvt}.)
\begin{multline}\label{bisGLperPar}
\Eper_\ang(\bfr,\nu,R,R',\theta)\\
=\underset{\psi\in \HCper}{\inf}\displaystyle{\int_{\mathbb{R}_+\times \DCper}\left(|(-i\nabla\qeb\mathbf{A}_{\nu,\ang})\psi|^2-\bfr|\psi|^2+\frac{\bfr}{2}|\psi|^4\right)\,\mathrm{d}x}.
\end{multline}
\begin{thm}\label{bislem:3.1Par}
Suppose that $\nu\in\left(0,\frac{\pi}{2}\right)$, that the magnetic flux satisfies \eqref{quantbis}, and that the linear spectral quantity $\zeta(\nu)$ from \eqref{zeta} satisfies
\begin{equation}\label{eq:assump_zeta}
 \quad \zv<\bfr<1.
\end{equation}

Then
\begin{align}\label{eq:Eigenvalues}
\zeta_\ang(\nu,R,R',\theta)=\zv,
\end{align}
and the quantity $\Eper_\ang(\bfr,\nu,R,R',\theta)$ defined in \eqref{bisGLperPar} is achieved in $\HCper$ with 
\begin{align}\label{eq:Energies}
\Eper_\ang(\bfr,\nu,R,R',\theta)<0.
\end{align}
In particular, minimizers of \eqref{bisGLperPar} exist and are non-trivial.
\end{thm}

Theorem \ref{bislem:3.1Par} states that the lowest eigenvalue of the spectral problem \eqref{3.5gen} is equal to the lowest eigenvalue of the following spectral problem on $\mathbb{R}^3_+$,
\begin{equation}\label{3.5genlin}
\left\{\begin{array}{ll}
(-i\nabla\qeb\mathbf{A}_{\nu,\ang})^2\psi=\lambda\psi & \ \text{in} \ \mathbb{R}^3_+\, ,\\
\frac{\partial\psi}{\partial x_1}=0 & \ \text{on} \ \p {\mathbb R}^3_+\, ,
\end{array}\right.
\end{equation}
which is $\zeta(\nu)$.
The rest of the section is devoted to the proof of Theorem \ref{bislem:3.1Par}.
\subsection{Proof of Theorem \ref{bislem:3.1Par}}\label{sec:proof}

\begin{proof}[Proof of Theorem \ref{bislem:3.1Par}]
By standard arguments, it suffices to prove \eqref{eq:Eigenvalues}.
Indeed, if $\mathcal{L}_{\ang}^{\mathsf{per}}(\nu,\theta) \psi = \zv \psi$, where as assumed $\zv < \bfr$, we get \eqref{eq:Energies} by using $\varepsilon \psi$ as a trial state in \eqref{bisGLperPar} for sufficiently small $\varepsilon$.

Now \eqref{eq:Eigenvalues} follows upon combining Lemma~\ref{aux} and Lemma~\ref{aux2} below. 
\end{proof}

So we proceed to establish the spectral estimates of Lemma~\ref{aux} and Lemma~\ref{aux2} below.

Recall \eqref{Av}. By a rotation by an angle $\ang$ with respect to the $x_1$ axis (which leaves the half-space $\RRh$ invariant), it is easy to see that the Schr\"odinger operator $\mathcal L(\nu)$ with constant magnetic field on the half-space (defined in \eqref{eq-3D-op-2}) and the following self-adjoint operator
\begin{equation}\label{eq:Ltv}
\mathcal L_\ang(\nu)=(-i\nabla\qeb\mathbf{A}_{\nu,\ang})^2\quad{\rm in}\ L^2(\R_+^3)\,,
\end{equation}
with domain
\begin{multline*}
D(\mathcal L_\ang(\nu))
=\{u\in L^2(\R_+^3)~:~ (-i\nabla\qeb\mathbf{A}_{\nu,\ang})u\in L^2(\R_+^3,\mathbb{C}^3)\,,\\
(-i\nabla\qeb\mathbf{A}_{\nu,\ang})^2u\in L^2(\R_+^3)\,, \;\; \partial_{x_1}u=0\text{ on } \p\mathbb R^3_+ \}\,,
\end{multline*}
are unitarily equivalent. Indeed, the geometrical domain remains unchanged after the rotation so that the spectrum does not change either. 
Clearly, this rotation also maps the eigenfunctions of the operator $\mathcal{L}(\nu)$ to the eigenfunctions of the operator $\mathcal{L}_\ang(\nu)$.

Thus, denoting $\zeta_\ang(\nu)$ the bottom of the spectrum of the operator $\mathcal L_\ang(\nu)$, we have that $\zeta_\ang(\nu)=\zv$.

We consider $\phidd$ defined in \eqref{shape}, with $\fou \in \mathscr{C}^\infty_c(\mathbb{R})$ (for concreteness, let us consider the same $\fou \in \mathscr{C}^\infty_c(\mathbb{R})$ as the one specified in the proof of Lemma \ref{Sdecay}). 

\begin{nota}\label{nota:phiddt}
Related to the mapping mentioned above, we denote by $\phiddt$ the
eigenfunction of the operator $\mathcal{L}_\ang(\nu)$. This eigenfunction is obviously  associated with the eigenvalue $\zeta(\nu)$.
\end{nota}

We define with $g_\latti$ from \eqref{eq:autre2}, the following function on $\RRh$
\begin{equation}\label{psibisgen}
\psi(x)=\underset{\latti\in \LLat}{\sum}\phiddt(x+\latti)e^{ig_\latti(x)}=\underset{\nn,\mm\in \Z}{\sum}\phiddt(x+\nn s+\mm t)e^{ig_{\nn s+\mm t}(x)}.
\end{equation}

\begin{lem}\label{lem:existencebisCNper}
For the function $\psi$ introduced in \eqref{psibisgen}, which is defined on ${\mathbb R}^3_+$, the following properties hold.
\begin{enumerate}[label=(\arabic*)]
\item \label{P1} The functions $\psi, (-i\nabla\qeb\AAT)\psi, (-i\nabla\qeb\AAT)^2\psi$ belong to $L^2(\mathbb{R}_+\times \Lat_{R,R',\theta})$
and $ (-i\nabla\qeb\AAT)^2\psi = \zeta(\nu)\psi$.
\item \label{P2} The function $\psi$ satisfies the magnetic periodic condition \eqref{3.4bis} and the Neumann boundary condition $\frac{\partial\psi}{\partial x_1}=0$ on the set $\p{\mathbb R}^3_+$.
\end{enumerate}
\end{lem}

\begin{proof}

We prove \ref{P1}. We first show that $\psi \in L^2(\mathbb{R}_+\times \Lat_{R,R',\theta})$.
It suffices to show that
\begin{equation}\label{but}
\underset{(j,k)\in\mathbb{Z}^2}{\sum}a_{j,k}<+\infty,
\end{equation}
with $
a_{j,k}= \Vert \phiddt(x+\nn s+\mm t)
\Vert^2_{L^2(\mathbb{R}_+\times \Lat_{R,R',\theta})}$.

Using the last assertion of Lemma \ref{Sdecay}, there exists a positive constant $C>0$ such that for all $x\in\RRh$
\begin{equation}\label{csq:exp}
(1+x_1^2)(1+x_2^2)(1+x_3^2)\big|\phiddt(x_1,x_2,x_3)\big|\le C.
\end{equation}

Inequality \eqref{but} follows from \eqref{csq:exp}. We omit the technical details.

Now $(-i\nabla\qeb\AAT)^2\psi \in L^2(\mathbb{R}_+\times \Lat_{R,R',\theta})$ since $(-i\nabla\qeb\AAT)^2\psi=\zeta(\nu)\psi$ (in the sense of distributions). We can obtain that the function $(-i\nabla\qeb\AAT)\psi$ belongs to $L^2(\mathbb{R}_+\times \Lat_{R,R',\theta},\C^3)$ proceeding in the same way as for the function $\psi$. Indeed, thanks to the last assertion of Lemma \ref{Sdecay}, we have in particular that all the first derivatives of $\psi$ satisfy an analogous inequality as in \eqref{csq:exp}, which is the central argument. This finishes the proof of \ref{P1}.

We prove \ref{P2}. The Neumann boundary condition at $x_1=0$ is satisfied for each term in the sum and is therefore clear. We prove the magnetic periodicity in the form of \eqref{eq:per_gen}, by calculating
\begin{align}
e^{ig_{j_0 s+k_0t}(x)} &\psi(x+j_0 s+k_0t) \nonumber \\
&=
\underset{\nn,\mm\in \Z}{\sum}\phiddt(x+(\nn+j_0) s+(\mm +k_0) t)e^{ig_{\nn s+\mm t}(x+j_0 s+k_0t)}
e^{ig_{j_0 s+k_0t}(x)}.
\end{align}
We need to verify that
\begin{align}
e^{ig_{\nn s+\mm t}(x+j_0 s+k_0t)}
e^{ig_{j_0 s+k_0t}(x)} = e^{ig_{(\nn+j_0) s+(\mm +k_0) t}(x)}.
\end{align}
But it follows by direct calculation using the definition \eqref{eq:autre2} of $g_\latti$ that
\begin{multline}
e^{ig_{\nn s+\mm t}(x+j_0 s+k_0t)}
e^{ig_{j_0 s+k_0t}(x)} e^{-ig_{(\nn+j_0) s+(\mm +k_0) t}(x)} \\
\quad = e^{ib((kt)\wedge(j_0s) + (js) \wedge (k_0 t))} = 1,
\end{multline}
where the last equality uses the flux quantization \eqref{quantbis}.
This finishes the proof of \ref{P2}.
\end{proof}

\begin{lem}\label{aux}
We have the following estimate:
\begin{equation}\label{eq1}
\zeta_\ang(\nu,R,R',\theta)\le\zv.
\end{equation}
\end{lem}
\begin{proof}

By Lemma \ref{lem:existencebisCNper}, $\psi$ is an eigenfunction of $\mathcal{L}_{\ang}^{\mathsf{per}}(\nu,\theta)$ with eigenvalue $\zeta(\nu)$.
\end{proof}

Recall the expression of the magnetic potential $\AAT$ defined in \eqref{Avgen}. Thanks to the translation invariance in the third variable $x_3$, and to the fact that the magnetic field $\mathbf{B}_{\nu,\ang}$ (see \eqref{Bvgen}) does not have any growth at infinity (it is constant), we can easily understand that the operator $\LLvpTT$ has essential spectrum. We will use Lemma \ref{aux} to prove the following result
\begin{lem}\label{eigenp}
The operator $\LLvpTT$ has an eigenvalue strictly less than the infimum of its essential spectrum.
\end{lem}
\begin{proof}{\it Step 1.} To prove that the operator $\LLvpTT$ has an eigenvalue below the threshold of its essential spectrum, it is enough to prove that
\begin{equation}\label{eq:goal}
\inf\spe_{\mathsf{ess}}\left(\LLvpTT\right) \ge 1.
\end{equation} Indeed, as we know that $\zeta(\nu)<1$ for all $\nu\in\left(0,\frac{\pi}{2}\right)$ (see Lemma \ref{lem-p-z(nu)}), Lemma \ref{aux} therefore implies that $\ZCper<1$ for all $\nu\in \left(0,\frac{\pi}{2}\right)$ (see Notation \ref{zetavRbeforeGEN}). Establishing \eqref{eq:goal} will give that the quantity $\ZCper$ belongs to the discrete spectrum of the operator $\LLvpTT $, giving in particular that it is an eigenvalue.

{\it Step 2}. We introduce the following self-adjoint operator $\mathcal{L}$, having formally the same expression as the operator $\LLvpTT$, but considered on the whole space $\mathbb{R}^3$. By rotational invariance of the operator $\mathcal L$, it is well known that its spectrum does not depend on $\nu$. This operator corresponds to the magnetic Laplacian with constant magnetic field and we know that (see Prop~\ref{prop:NoBoundary})
\begin{equation}\label{valueOne}\inf\spe(\mathcal{L})=\inf\spe_{\mathsf{ess}}(\mathcal{L})=1.
\end{equation}

We choose $L>0$ (to be chosen large at the end of the proof) and we denote by $\LLvpDLT$ the self-adjoint realization of the linear operator $\LLvpTT$ on $\HCper$, on the subset of functions vanishing on $\le x_1 \le L$.
Namely, the domain of $\LLvpDLT$ is the set
\begin{multline*}
\{\psi\in L^2_{\rm comp}((L,+\infty)\times\R^2)~:~ (-i\nabla+\AAT)\psi\in L^2_{\rm comp}((L,+\infty)\times\R^2,{\mathbb C}^3),\; \\ (-i\nabla+\AAT)^2\psi\in L^2_{\rm comp}((L,+\infty)\times\R^2), \psi\;\text{satisfies \eqref{3.4bis}},\;  \psi=0\;\;\text{on }\; x_1=L\}.
\end{multline*}
where we recall that $\HCper$ has been introduced in \eqref{eq:Hvt}.

Using Persson's Theorem (\cite{Per60}, see also \cite[Theorem B.1.1]{FH}), we have
\begin{equation}\label{step1a}
\inf\spe_{\mathsf{ess}}\left(\LLvpTT\right) =\underset{L\rightarrow+\infty}{\lim}\inf\spe\left(\LLvpDLT\right).
\end{equation}
By inclusion of the domains, we have
\begin{equation}\label{step1b}
\underset{L\rightarrow+\infty}{\lim}\inf\spe\left(\LLvpDLT\right)\ge\inf\spe\left(\LLvpDzT\right),
\end{equation}
where $\LLvpDzT$ is the operator on $\R^3_+$ with the same expression as $\LLvpTT$, and with Dirichlet boundary condition at $\p \R^3_+$. 

{\it Step 3.} Using \eqref{valueOne},
\eqref{step1a} and \eqref{step1b},
we need to prove that
\begin{equation}\label{eq:goal2}
\lambda_0 := \inf\spe\left(\LLvpDzT\right) \ge \inf\spe(\mathcal{L})=1.
\end{equation}

To prove \eqref{eq:goal2} we construct a sequence $(\psiN)$ of functions in the domain of the operator $\mathcal{L}$ such that
\begin{equation}\label{limz}
\underset{n\rightarrow+\infty}{\lim}\frac{\mathcal{Q}(\psiN)}{\Vert\psiN\Vert^2_{L^2(\mathbb{R}^3)}}= \lambda_0,
\end{equation}
where $\mathcal{Q}$ is the quadratic form associated with the operator $\mathcal{L}$. We know that there exists a minimizing sequence $(\psi_n)_{n\in\mathbb{N}^*}$ such that
\begin{equation}\label{minseq}
\underset{n\rightarrow+\infty}{\lim}\frac{\QQvpDzT(\psi_n)}{\Vert\psi_n\Vert_{L^2(\Cfon)}}=\lambda_0,
\end{equation}
where $\QQvpDzT$ is the quadratic form associated with the operator $\LLvpDzT$.
By density of the set of smooth functions with support bounded with respect to $x_1$, we can assume without loss of generality that $\psi_n$ is smooth and for each $n$ there exists $k$ such that
\begin{equation}\label{condsuppcomp2}
\supp (\psi_n) \subset \{x=(x_1,x_2,x_3)\in \R^3~:~ |x_1|\le k\}.
\end{equation}
We let $k(n)$ be the smallest such integer $k$.
For all $n\in\mathbb{N}^*$, we denote by $\tilde{\psi}_n$ the function $\psi_n$ extended by $0$ to the complement of $\RRh$ so that for all $n\in\mathbb{N}^*$, the function $\tilde{\psi}_n$ is defined on the whole space $\mathbb{R}^3$. 

We consider a sequence of cutoff functions $(\chi_n)_{n\in\mathbb{N}^*}\subset\mathscr{C}_c^{\infty}(\mathbb{R})$ satisfying
\begin{equation}\label{cutoff}
0\le \chi_n\le 1\;\; \text{on}\;\; \R, \quad \chi_n= 1 \;\; \text{on} \;\; [-n,n], \quad \chi_n= 0 \;\; \text{on} \;\; \mathbb{R}\backslash [-(n+1),n+1],
\end{equation}
and
\begin{equation}\label{cutoff2}
|\nabla\chi_n|\le C,
\end{equation}
for all $n\in\mathbb{N}^*$ and where $C>0$ is a fixed constant. Using these cutoff functions, we build the following sequence of real valued functions $\eta_n$ in $\mathscr{C}_c^{\infty}(\mathbb{R}^3)$ defined (for all $n\in\mathbb{N}^*$) on $\mathbb{R}^3$ as follows
\begin{equation}\label{eq:cutoff_eta}
\eta_{n}(x)=\chi_{k(n)}(x_1)\chi_n\left(\frac{x_2}{R}\right)\chi_{n}\left(\frac{x_3}{R'\sin\theta}\right).
\end{equation}
Notice that the cutoff function $\eta_{n}$ defined in \eqref{eq:cutoff_eta} depends on $\theta$, $R$, $R'$ and the function $k$ but we have chosen not to indicate this dependence to simplify the notations.

For all $n\in\mathbb{N}^*$, we define $\psiN$ as
\begin{equation*}
\psiN=\eta_n\tilde{\psi}_n.
\end{equation*}
Thanks to the cutoff, the functions $\eta_n\tilde{\psi}_n$ belong to the domain of the operator $\mathcal{L}$, for all $n\in\mathbb{N}^*$. What is more, thanks the choice of $k(n)$ we have
\begin{equation}\label{condsuppcomp3}
\eta_n(x)\tilde{\psi}_n(x)=\chi_n\left(\frac{x_2}{R}\right)\chi_n\left(\frac{x_3}{R'\sin\theta}\right)\tilde{\psi}_n(x).
\end{equation}

We introduce the following sets (for $l \in {\mathbb N}$)
\begin{equation*}
\Cfonbis=(0,+\infty)\times \DCper,
\end{equation*}
\begin{equation*}
\mathcal{C}_l=(0,+\infty)\times\{(x_2,x_3)\in\R^2~:~ |x_2|< lR,\;\; |x_3|< lR'\sin\theta\},
\end{equation*}
\begin{multline*}
\mathcal{K}_l=(0,+\infty) \times\{(x_2,x_3)\in \R^2~:~ lR< |x_2|< (l+1)R,\;\;  \text { or }\\
lR'\sin\theta< |x_3|< (l+1)R'\sin\theta\}.
\end{multline*}
Notice that these sets depend on $\theta$, $R$ and $R'$ but we have chosen not to encumber the notations.

Define
\begin{align*}
I_l&:= \{ w \in \LLat_{R,R',\theta}\,:\, \Cfonbis + w \subset \mathcal{C}_l\},\\
\tilde{I_l} &:= \{ w \in \LLat_{R,R',\theta}\,:\, (\Cfonbis + w )\cap  \mathcal{K}_l \neq \emptyset\}.
\end{align*}
We will use that 
\begin{align}\label{eq:surface}
|\tilde{I_l} |/|I_l| \rightarrow 0, \qquad \text{ as } \quad l\rightarrow \infty,
\end{align}
which easily follows by an area consideration.

We now calculate/estimate using the periodicity of $\tilde{\psi}_n$ and the definition of $\eta_n$
\begin{align}\label{eq:energy}
&\left| \int |(-i\nabla\qeb\mathbf{A}_{\nu,\ang})\psiN |^2\,dx - |I_n| \int_{\Cfonbis}  |(-i\nabla\qeb\mathbf{A}_{\nu,\ang}) \tilde{\psi}_n |^2\,dx \right| \nonumber \\
&\leq \int_{\RRh \setminus (\cup_{w \in I_n} \Cfonbis + w)} |(-i\nabla\qeb\mathbf{A}_{\nu,\ang})\psiN |^2\,dx\nonumber \\
&\leq \int_{\cup_{w \in \tilde{I}_n} (\Cfonbis + w)} |(-i\nabla\qeb\mathbf{A}_{\nu,\ang})\psiN |^2\,dx\nonumber \\
&\leq C |\tilde{I}_n| \int_{\Cfonbis}  |(-i\nabla\qeb\mathbf{A}_{\nu,\ang}) \tilde{\psi}_n |^2 + |\tilde{\psi}_n |^2\,dx.
\end{align}
Similarly,
\begin{align}\label{eq:mass}
\left| \int |\psiN |^2\,dx - |I_n| \int_{\Cfonbis}  |\tilde{\psi}_n |^2\,dx \right| 
&\leq |\tilde{I}_n| \int_{\Cfonbis}  | \tilde{\psi}_n |^2 \,dx.
\end{align}
From \eqref{eq:mass} we get using \eqref{eq:surface} that
\begin{align}\label{eq:mass2}
\left| \frac{ \int |\psiN |^2\,dx}{ |I_n| \int_{\Cfonbis}  |\tilde{\psi}_n |^2\,dx} -1 
\right| \rightarrow 0, \qquad \text{ as } n\rightarrow \infty.
\end{align}
Similarly, from \eqref{eq:energy} we find using \eqref{eq:surface} and \eqref{minseq} that
\begin{align}\label{eq:energy2}
\frac{\int |(-i\nabla\qeb\mathbf{A}_{\nu,\ang})\psiN |^2\,dx}{ |I_n| \int_{\Cfonbis}  |\tilde{\psi}_n |^2\,dx} \rightarrow \lambda_0.
\end{align}
Combining \eqref{eq:energy2} and \eqref{eq:mass2} we conclude \eqref{limz} which finishes the proof.
\end{proof}

Since we know by Lemma \ref{eigenp} that $\ZCper$ is an eigenvalue strictly below the essential spectrum we have the following classical result (see for example \cite[Theorem B.5.1]{FH}).

\begin{lem}\label{exp_decay_x1_p}
Let $\psivT$ be a normalized eigenfunction of the operator $\mathcal{L}_{\ang}^{\mathsf{per}}(\nu,\theta)$, associated with the lowest eigenvalue $\ZCper$. There exist constants $C, \alpha>0$ such that:
\begin{equation*}
\displaystyle{\int_{\mathbb{R}_+\times \Lat_{R,R',\theta}}e^{\alpha x_1}\left(\big|\psivT(x)\big|^2 + \big|(-i\nabla\qeb\mathbf{A}_{\nu,\ang})\psivT(x) \big|^2\right)\,\mathrm{d}x}<C.
\end{equation*}
\end{lem}

\begin{lem}\label{aux2}
We have \begin{equation*}\label{eq:OtherBound}
\ZCper\ge \zv.
\end{equation*}
\end{lem}

\begin{proof}
To prove \eqref{eq:OtherBound} it suffices to construct a sequence $(\psiN)$ of functions in the domain of the operator $\LLv$, such that
\begin{equation}\label{limzbis}
\underset{n\rightarrow+\infty}{\lim}\frac{\mathcal{Q}_\nu(\psiN)}{\Vert\psiN\Vert^2_{L^2(\mathbb{R}^3_+)}}= \zeta_\ang(\nu,R,R',\theta),
\end{equation}
where $\mathcal{Q}_\nu$ is the quadratic form associated with the operator $\mathcal{L}(\nu)$.

Let $\psivT \in \HCper$ be an eigenfunction of the operator $\mathcal{L}_{\ang}^{\mathsf{per}}(\nu,\theta)$, associated with the lowest eigenvalue $\ZCper$ (see Lemma \ref{eigenp}).
In particular,
\begin{align}
\frac{\QQvpT(\psivT)}
{\| \psivT \|_{L^2(\mathbb{R}_+\times \DCper)}^2} = \ZCper.
\end{align}
We consider a sequence $(\chi_n)_{n\in\mathbb{N}^*}$ of cutoff functions $(\chi_n)\subset\mathscr{C}_c^{\infty}(\mathbb{R})$ satisfying \eqref{cutoff} and \eqref{cutoff2}.
Using these cutoff functions, we define $\eta_n$ in $\mathscr{C}_c^{\infty}(\RRh)$ by
\begin{equation*}
\eta_n(x)=\chi_n\left(x_2\right)\chi_n\left(x_3\right).
\end{equation*}

Thanks to the cutoff in the variables $x_2$ and $x_3$, and the decay in the variable $x_1$, the functions $\psiN:=\eta_n\psivTt$ belong to the domain of the operator $\LLv$, for all $n\in\mathbb{N}^*$. 
One can now prove that $(\psiN)$ satisfies \eqref{limzbis} using the exact same arguments as in the proof of Lemma~\ref{eigenp}. We omit the details.
\end{proof}

\end{document}